\documentclass[leqno,12pt]{article} 
\usepackage{amsmath, amssymb}
\usepackage{amsthm} 
\usepackage{epsfig}
\usepackage{lipsum}
\usepackage{graphicx}
\usepackage{calc}
\usepackage{gensymb}
\usepackage{comment}
\usepackage{wrapfig}
\usepackage[sort, numbers]{natbib}
\usepackage{placeins}
\usepackage{hyperref}
\usepackage[dvipsnames]{xcolor}
\usepackage{ulem}
\usepackage{authblk}
\usepackage{etex}
\usepackage[etex=true,export]{adjustbox}
\usepackage{float}
\usepackage{xkeyval}
\usepackage{mathtools}
\usepackage{amsmath}
\usepackage{amsthm}
\usepackage{amsfonts}
\usepackage{etoolbox}
\usepackage{hyperref}
\usepackage{amssymb}
\usepackage{array}
\usepackage{multirow}
\usepackage{lipsum}
\usepackage{hyperref}
\usepackage{authblk}
\usepackage{url}
\newcommand\marklessfootnote[1]{
    \addtocounter{footnote}{1} 
    \footnotetext{#1}
}
\theoremstyle{plain} 
\newtheorem{theorem}{\indent\sc Theorem}[section]
\newtheorem{lemma}[theorem]{\indent\sc Lemma}
\newtheorem{corollary}[theorem]{\indent\sc Corollary}
\newtheorem{proposition}[theorem]{\indent\sc Proposition}

\theoremstyle{definition} 
\newtheorem{definition}[theorem]{\indent\sc Definition}
\newtheorem{remark}[theorem]{\indent\sc Remark}
\newtheorem{question}[theorem]{\indent\sc Question}
\newtheorem{example}[theorem]{\indent\sc Example}

\title{\uppercase{\normalsize A New Polynomial for Checkerboard-Colorable 4-Valent Virtual Graphs}}
\author
{Hamid Abchir$^{1*}$, Khaled Qazaqzeh$^2$, and Mohammed Sabak$^3$\\
\scriptsize{$^1$Fundamental and Applied Mathematics Laboratory, Hassan II University, Casablanca, Morocco, 20100
\\
$^2$Department of Mathematics, Faculty of Science, Yarmouk University, Irbid, Jordan, 21163 \\
$^3$Fundamental and Applied Mathematics Laboratory, Hassan II University, Casablanca, Morocco, 20100}
}

\begin{document}
\maketitle
\marklessfootnote{
}
\marklessfootnote{Key words and phrases.
\textit{4-valent graphs, Jones-Kauffman polynomial, Euler circuits, virtual links, checkerboard-colorable.}}

\begin{abstract} 
We assign a new polynomial to any vertex-signed checkerboard-colorable 4-valent virtual graph in terms of its Euler circuit expansion. This provides a new combinatorial formulation of the Jones-Kauffman polynomial for checkerboard-colorable virtual links.
\end{abstract}

\section{Introduction}
The Jones-Kauffman polynomial is an invariant for oriented virtual links, first introduced in \cite{Ka} and inspired by the work in \cite{J, Ka2}. Following its introduction, numerous combinatorial formulations of this polynomial have been developed, initially for classical links and later generalized for virtual links. In the late 1980s, Thistlethwaite \cite{Th} expressed the Jones-Kauffman polynomial of a classical link using an improved version of the Tutte polynomial of a planar graph associated with the link. In the late 2000s, several authors adopted a similar strategy, expressing the Jones-Kauffman polynomial of classical and virtual links in terms of the Bollobás–Riordan polynomial. This polynomial, introduced in \cite{br}, generalizes the Tutte polynomial for ribbon graphs. This approach using the Bollobás–Riordan polynomial was first introduced in \cite{cp} for checkerboard-colorable virtual links, followed by two additional realizations in \cite{cv} for general virtual links, and in \cite{dfk} for classical links. Although the results presented in \cite{cp}, \cite{cv}, and \cite{dfk} are formally distinct, each of them uses a different construction of ribbon graphs from link diagrams and utilizes varying substitutions in the Bollobás-Riordan polynomials for these graphs. Eventually, the work in \cite{S. Chmutov} unified these different results by introducing a new duality concept for graphs embedded in surfaces. Recently, in 2017, Q. Deng et al. \cite{Q. Deng} introduced the concept of cyclic graphs, which is equivalent to orientable ribbon graphs. Then they defined a polynomial for these specific graphs, which is related to the Jones-Kauffman polynomial and provided a new expression of this polynomial. 

In this paper, we follow a similar approach and give a new combinatorial formulation of the Jones-Kauffman polynomial specifically for checkerboard-colorable virtual links. We start by defining a new polynomial invariant for vertex-signed checkerboard-colorable 2-digraphs based on Euler circuits. Then, we show that this polynomial satisfies a skein relation, which enables us to recover the Jones-Kauffman polynomial of checkerboard-colorable virtual links.

Now we describe the contents of this paper. In Section 2, we briefly introduce and recall the main terms that will be used in this paper. In Section 3, we define the new polynomial for vertex-signed 2-digraphs. Also, we state the main theorem, and we give some properties of this polynomial. In particular, we show that it satisfies a skein relation. At the end of this section, we give an application that shows that the polynomial of the signed-vertex 2-digraph associated to a virtual link diagram allows us to recover its Jones-Kauffman polynomial. The proofs of the primary theorems are addressed in Section 4.
\section{Preliminary}

We study 4-valent planar connected graphs with two types of double points, \textit{actual} double points which constitute the vertices, and \textit{virtual} double points. In other words, the edges of such graphs can intersect transversely in some points which are not vertices called virtual double points. Such graphs are called {\sl 4-valent virtual graphs}. An edge in such graph is a curve joining two consecutive vertices. A \textit{$2$-digraph} is a directed $4$-valent virtual graph such that at each vertex there are $2$ incoming and $2$ outgoing edges. 

We say that a $4$-valent virtual graph $G$ is \textit{checkerboard-colorable} if there is a coloring of one side of each edge in the graph such that near a vertex the coloring of opposite edges alternate and near a virtual double point the colorings go through independent of the crossing strand and its coloring. See Figure \ref{fig2} as an example of a checkerboard-colored 4-valent virtual graph.

\begin{figure}[H]
\centering
\includegraphics[scale=0.23]{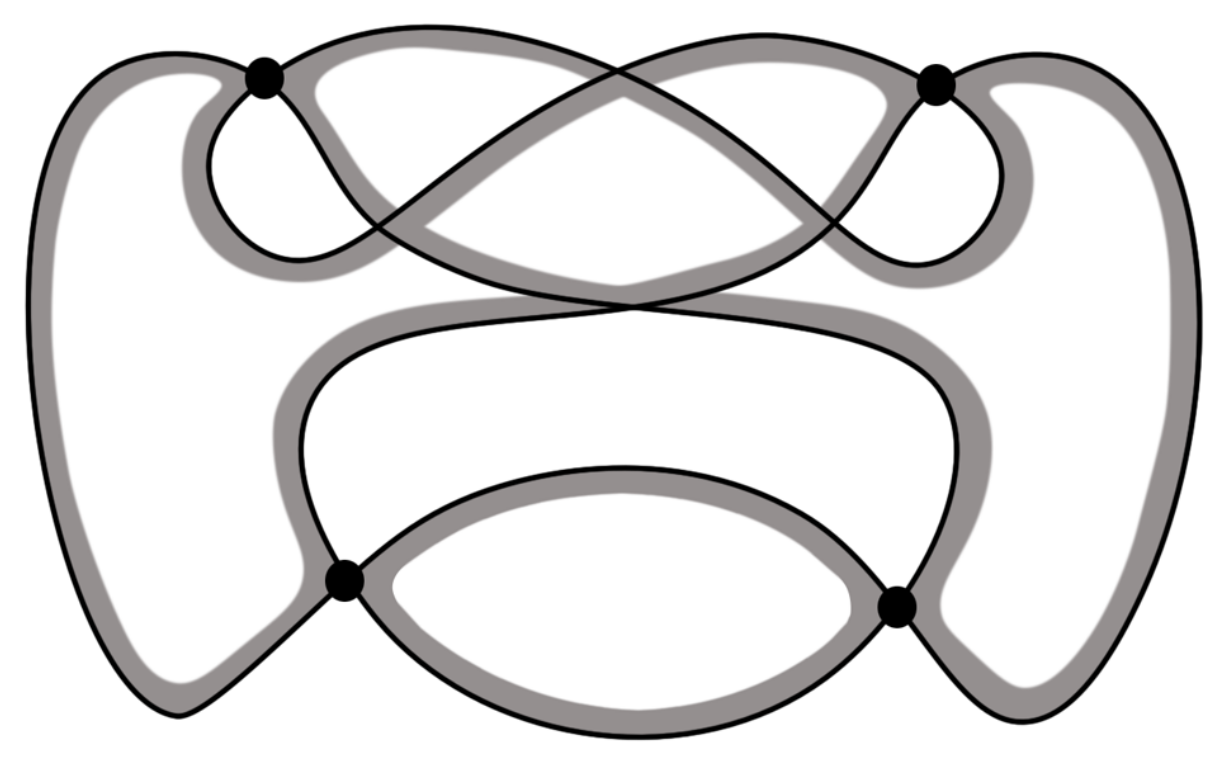}   
\caption{Checkerboard-colored 4-valent virtual graph.}
\label{fig2}
\end{figure}

A 2-digraph is said to have a \textit{source-target} structure if the orientation of the incoming and outgoing edges at each vertex is as depicted in Figure \ref{fig1}. Note that for a given $2$-digraph, either it has no source-target structure or it has exactly two source-target structures, each of which is obtained from the other by reversing the orientation of each edge.
\begin{figure}[H]
\centering
\includegraphics[scale=0.13]{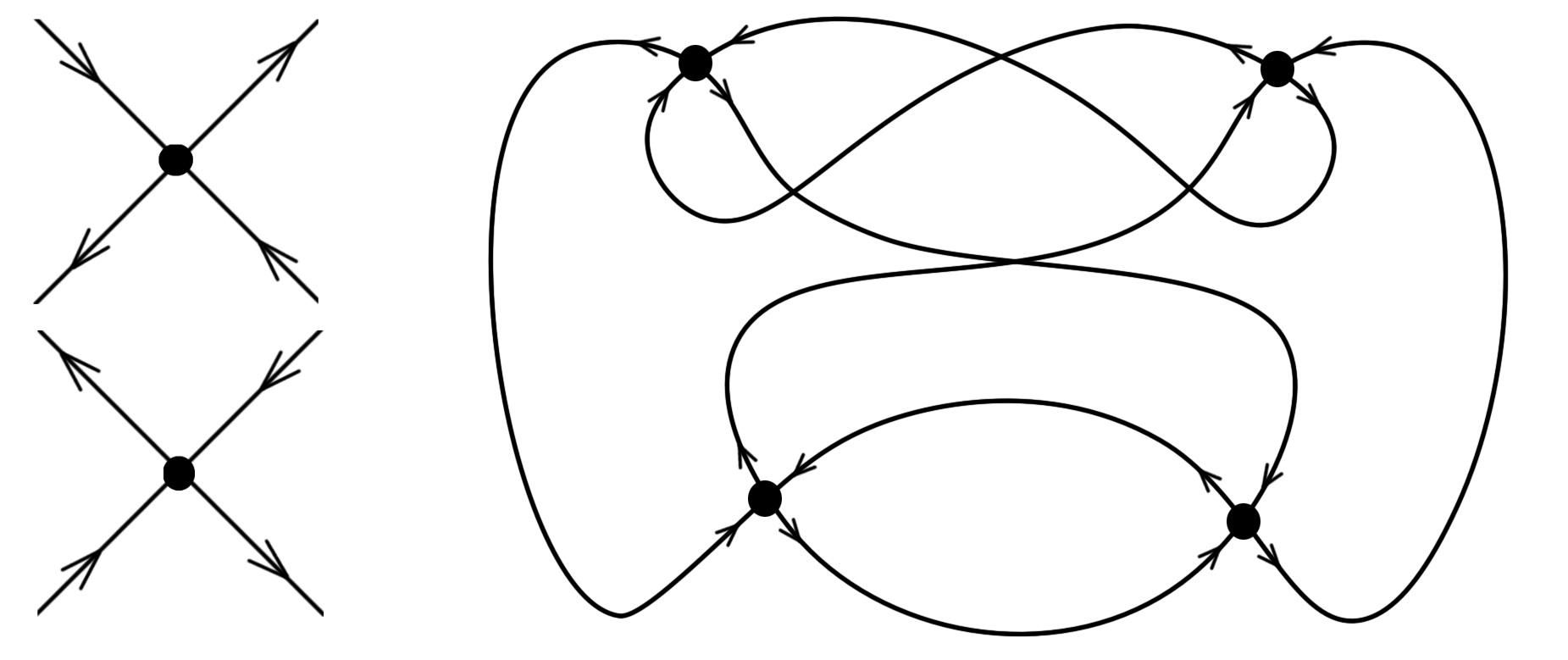}   
\caption{Allowed orientations around each vertex for a 2-digraph to have a source target structure.}
\label{fig1}
\end{figure}

\begin{remark}
A 2-digraph has a source-target structure iff it is checkerboard-colorable. In fact, if $G$ is a 2-digraph with a source-target structure, then to each oriented edge we associate a normal vector satisfying the right hand rule in the plane as depicted in the left of Figure \ref{fig5}. By coloring the side of each edge containing its normal vector we get a checkerboard coloring of $G$. Conversely, If $G$ is a checkerboard colored 4-valent graph, then, for each edge, we put a normal vector inside its colored side with initial point on that edge as depicted in the right of Figure \ref{fig5}. Then we orient the edge by using the right hand rule.
\end{remark}
Throughout the remainder of this paper, we adopt the convention that a checkerboard coloring is associated with a source-target structure as described above.

\begin{figure}[H]
\centering
\includegraphics[scale=0.18]{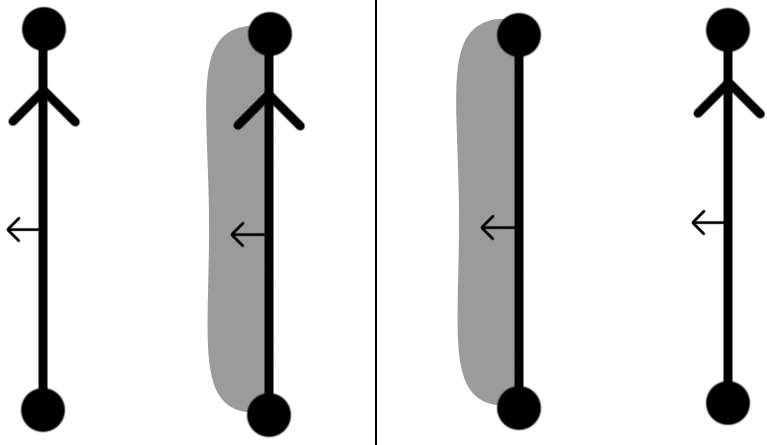}   
\caption{Determining a checkerboard coloring of a 4-valent graph from a source target structure and vice versa.}
\label{fig5}
\end{figure}

A 4-valent graph $G$ is said to be \textit{checkerboard-embedded} in some orientable surface $\Sigma$ if there exists an embedding $e$ of $G$ in $\Sigma$ such that the components of $\Sigma\setminus e(G)$ are 2-cells and there exists a coloring of these components by two colors, say black and white, such that two components of $\Sigma\setminus e(G)$ which are adjacent by an edge of $e(G)$ have always distinct colors. Manturov in \cite{manturov,m} showed that a 4-valent virtual graph is checkerboard-colorable iff it is checkerboard-embeddable in some orientable surface.

\subsection{Euler circuits and chord diagrams}

An \textit{Euler circuit} of a digraph is a path satisfying the following properties: 

\begin{enumerate}
\item[•] It begins and ends at the same vertex, forming a closed loop.
\item[•] It traverses each directed edge of the graph once.
\end{enumerate}

A digraph is said \textit{Eulerian} if it has an Euler circuit. It is known that a digraph is Eulerian if and only if the number of incoming and outgoing edges at each vertex are equal \cite{arratia}. Thus, every 2-digraph is Eulerian and each of its Euler circuits visits each vertex exactly twice. 

In the graph $G$, we let $v_1, \hdots , v_n$ be the set of vertices of such graph and $\gamma$ be an Euler circuit of $G$. Draw a simple circle $C$ in the plane and mark $2n$ evenly spaced points around its circumference. Start traveling $\gamma$ following the orientation of the edges starting from any vertex. Label each of the $2n$ points on the circle $C$ successively with the label of the vertex we encounter. Each vertex label appears twice at exactly two distinct points on the circle. For each vertex $v_i$ in $G$, draw a chord inside the circle $C$ connecting the two points labeled $v_i$. We ensure that each pair of chords intersect at most once and no more than two chords intersect at the same point. The circle $C$ together with the $n$ chords constitute the \textit{chord diagram} of $\gamma$, which we denote by $C(\gamma)$. See Figure \ref{fig3} as an example for an Euler circuit $\gamma$ and its corresponding chord diagram $C(\gamma)$. 

\begin{figure}[H]
\centering
\includegraphics[scale=0.17]{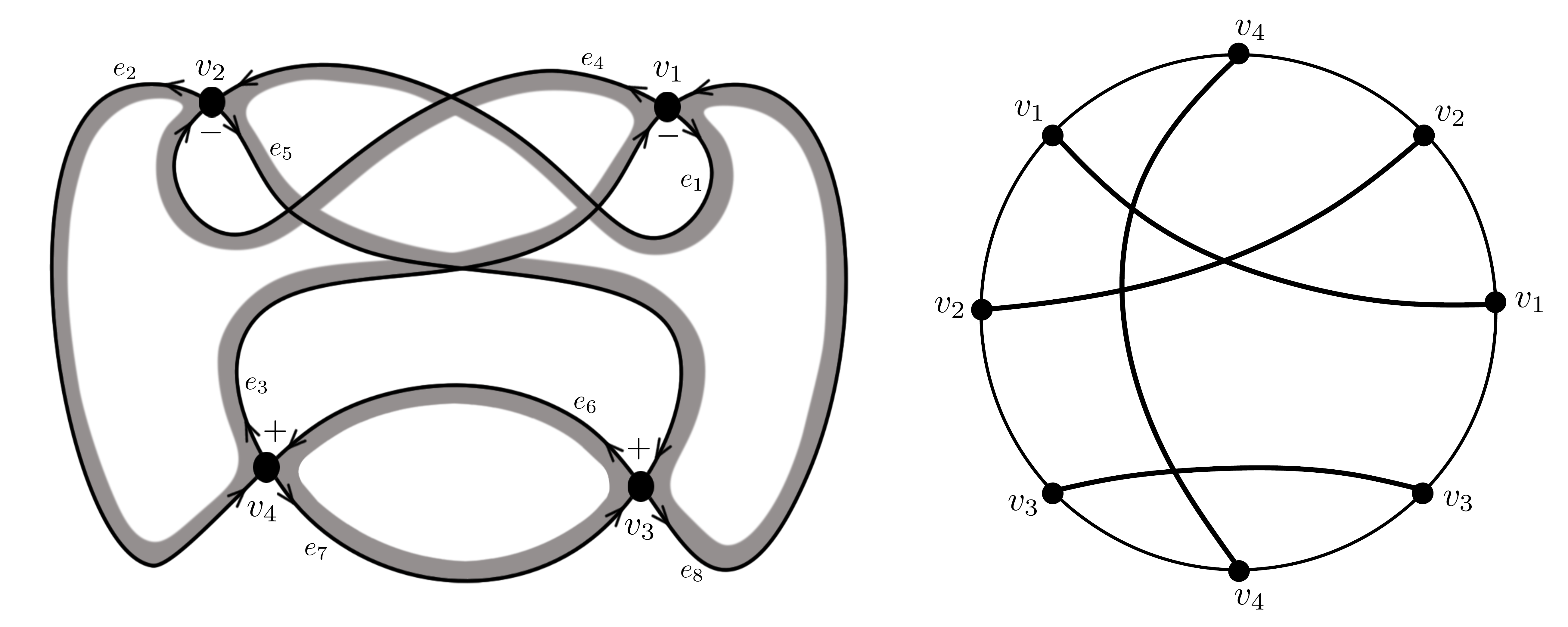}   
\caption{A 2-digraph with an Euler circuit $\gamma = v_{1}e_{1}v_{2}e_{2}v_{4}e_{3}v_{1}e_{4}v_{2}e_{5}v_{3}e_{6}v_{4}e_{7}v_{3}e_{8}v_{1}$ on the left and the corresponding chord diagram of $\gamma$ on the right.}
\label{fig3}
\end{figure}

We say that two distinct vertices $v_i$ and $v_j$ of $G$ \textit{interlace} in $\gamma$ if their corresponding chords intersect in the chord diagram $C(\gamma)$. For a vertex $v_i$ of $G$, we denote the subset of $\left\lbrace 1,2, \hdots ,n \right\rbrace$ consisting of the indices of vertices interlacing with $v_i$ in $\gamma$ by $\mathcal{C}_i(\gamma)$. We denote the complement of $\mathcal{C}_{i}(\gamma)$ in $\left\lbrace 1, \hdots , n \right\rbrace$ by $\overline{\mathcal{C}}_{i}(\gamma)$.

\subsection{Virtual links}
A \textit{virtual link diagram} is an immersion of one or several circles in the plane with only double points, such that each double point is either a \textit{classical crossing} indicated by over and under-crossings, or a \textit{virtual crossing} indicated by a circle as depicted in Figure \ref{fig10}. Two virtual link diagrams are \textit{equivalent} if they are related by a finite sequence of the generalized Reidemeister moves which are local modifications of virtual link diagrams depicted in Figure \ref{fig11}.

\begin{figure}[H]
\centering
\includegraphics[scale=0.2]{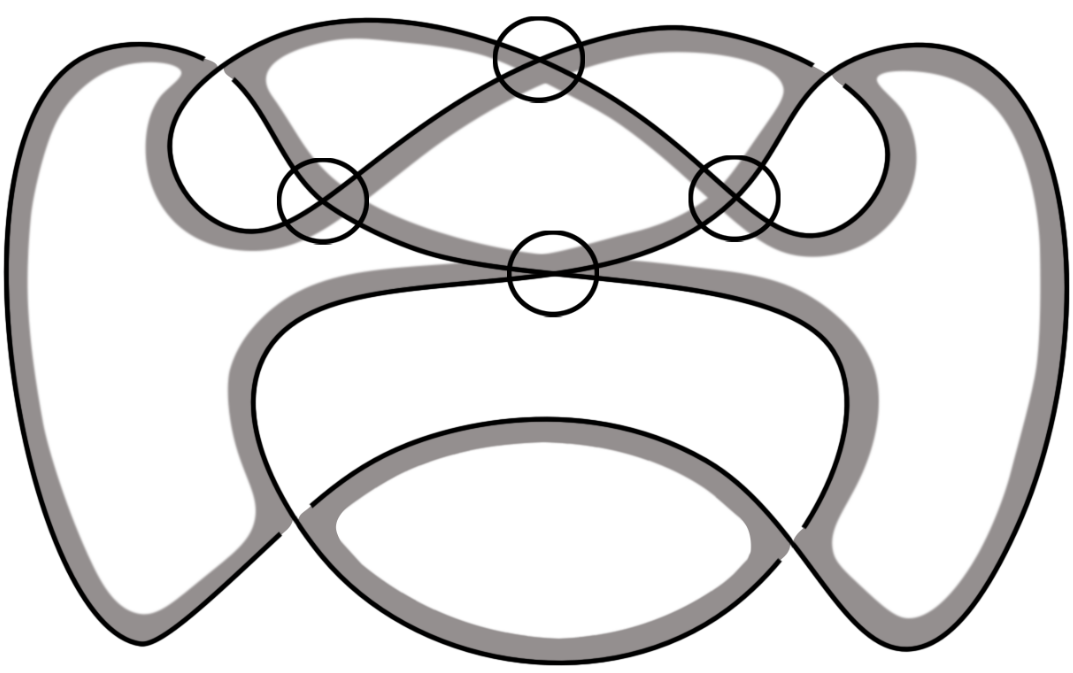}   
\caption{A checkerboard-colorable virtual link diagram.}
\label{fig10}
\end{figure}

\begin{figure}[H]
\centering
\includegraphics[scale=0.25]{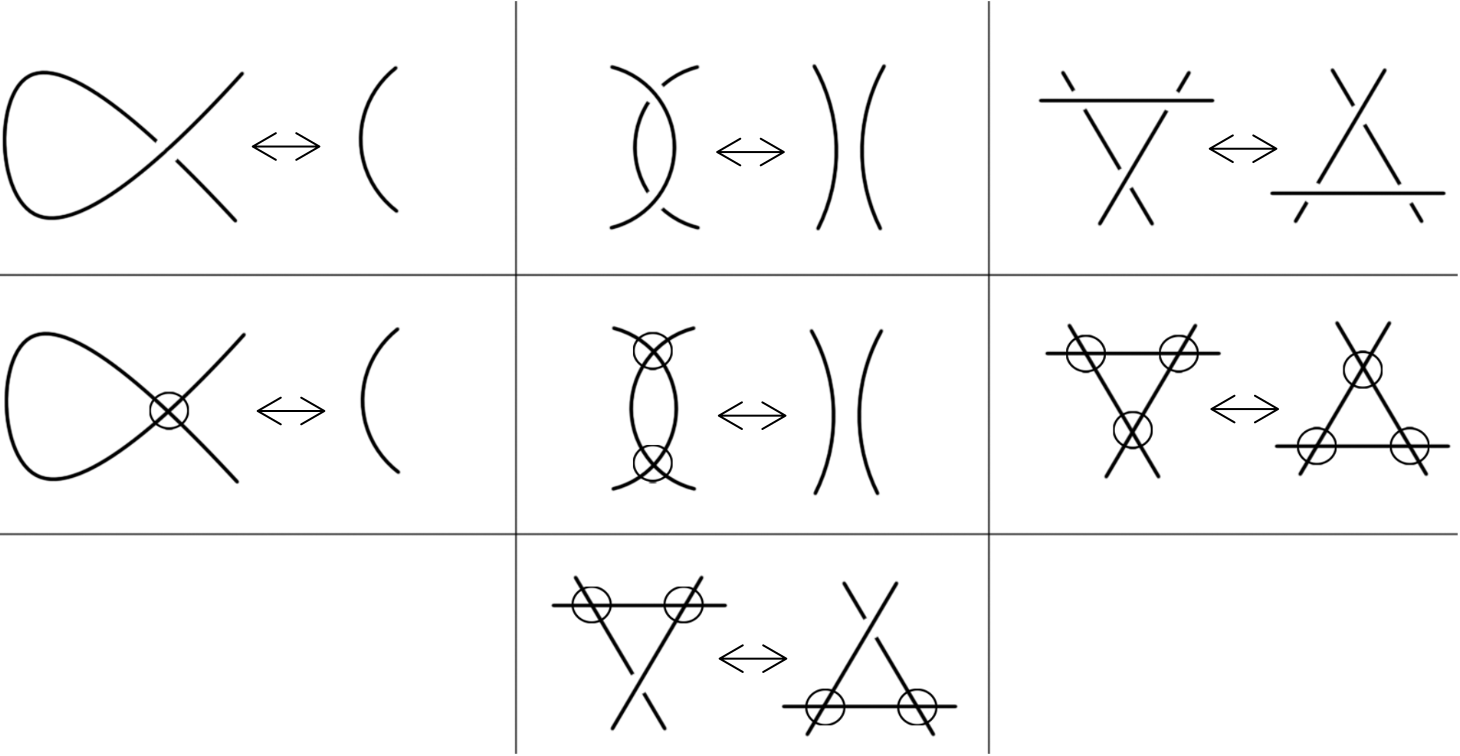}   
\caption{The generalized Reidemeister moves.}
\label{fig11}
\end{figure}

A \textit{virtual link} is an equivalence class of virtual diagrams under the generalized Reidemeister moves. These links were introduced by Kauffman in \cite{Ka} as a natural generalization of classical links.

\subsection{The Jones-Kauffman polynomial of a virtual link}
One of the most famous invariants of oriented virtual links is the Jones-Kauffman polynomial which was first introduced in \cite{Ka} and motivated by the work in \cite{J,Ka2}. It is a Laurent polynomial with integral coefficients that can be defined in terms of the Kauffman bracket. 
We briefly recall the definition of this polynomial as follows:

\begin{definition}
The Kauffman bracket polynomial is a function from the set of unoriented virtual link diagrams in the oriented plane to the ring of Laurent polynomials with integer coefficients in an indeterminate $q$. It maps a virtual link diagram $L$ to $\left\langle L \right\rangle\in \mathbb Z[q^{-1},q]$ and it is defined by the following relations:
\begin{enumerate}
\item $\left\langle \bigcirc \right\rangle=1$,
\item $\left\langle \bigcirc \cup L\right\rangle=(-q^{-2}-q^2)\left\langle L
\right\rangle$,
\item $\left\langle L\right\rangle=q\left\langle L_0\right\rangle+q^{-1}\left\langle
L_1\right\rangle$,
\end{enumerate}
where $\bigcirc$, in relation 2,  denotes a trivial circle  disjoint from the rest of the virtual link, and  $L,L_0, \text{and } L_1$  represent three unoriented virtual link diagrams which are identical everywhere  except in a small region where they are as indicated in  Figure \ref{figure}.
\end{definition}
      \begin{figure} [h]
\begin{center}
\includegraphics[scale=0.3]{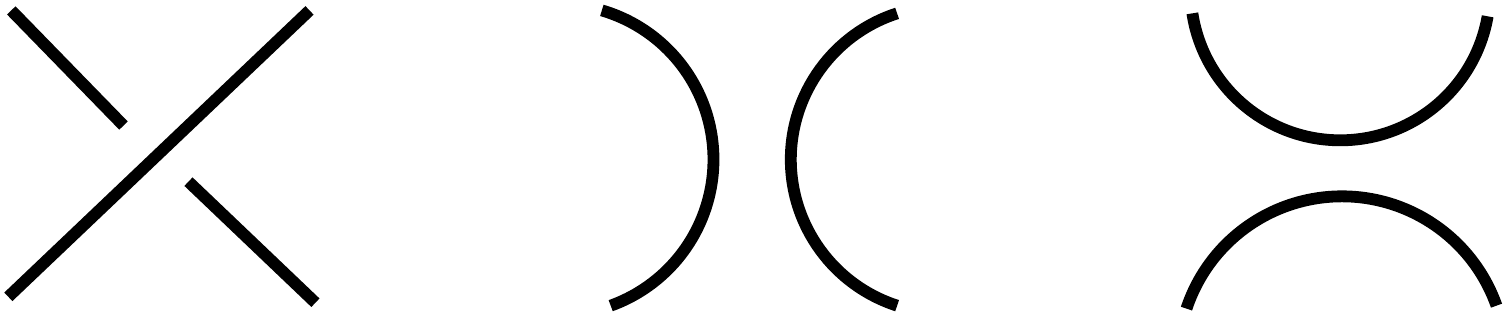} \\
\hspace{0.1cm}{$L$}\hspace{2.7cm}{$L_0$}\hspace{2.5cm}$L_{1}$
\end{center}
\vspace{-0.4cm}
\caption{The link diagram $L$ at the crossing $c$ and its smoothings $L_{0}$ and $L_{1}$.}\label{figure}
\end{figure}

For an oriented virtual link diagram $L$, we let $x(L)$ denote the number of negative crossings and $y(L)$ denote the number of positive crossings in $L$ according to the scheme in Figure \ref{Diagram1}. The writhe of the virtual link diagram $L$ is defined to be the integer $\omega(L)= y(L) - x(L)$.
\begin{definition}
The Jones-Kauffman polynomial $f_{L}(q)$ of the oriented virtual link $L$ is the Laurent polynomial in $q$ with integer coefficients defined by
\begin{equation*}
f_{L}(q)=(-q)^{-3\omega(L)}\left\langle L \right\rangle \in \mathbb Z[q,q^{-1}],
\end{equation*}
where $\left\langle L \right\rangle $ is the Kauffman bracket of the unoriented virtual link obtained from $L$ by ignoring  the orientation.
\end{definition}

\begin{figure}[h]
	\centering
		\reflectbox{\includegraphics[scale=0.1]{Diagram1}}\hspace{1.1cm}
		\includegraphics[scale=0.1]{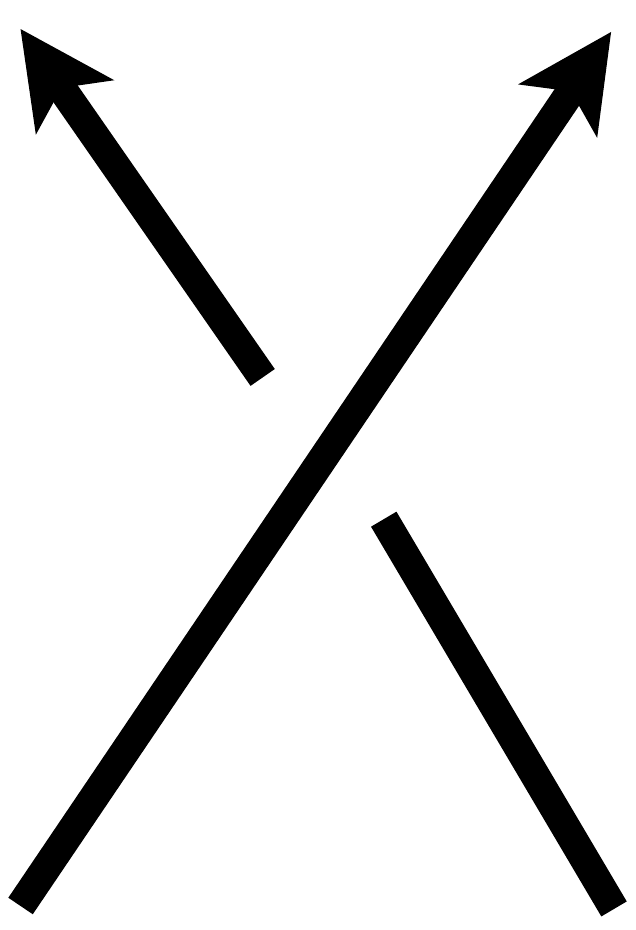}
		\vspace{-0.2cm}
	\caption{Negative and positive crossings respectively.}
	\label{Diagram1}
      \end{figure}

\section{The main results and an application}
We say that a $2$-digraph is \textit{vertex-signed} if each vertex is equipped with a sign $+$ or $-$.\\
For the rest of this paper, we assume that any $2$-digraph is vertex-signed and checkerboard-colorable unless mentioned otherwise.\\
Let $G$ be a $2$-digraph with $n$ vertices $v_i$, $1\le i\le n$, and $\gamma$ be an Euler circuit of $G$. At each vertex $v_i$, we denote the two incoming edges by $e_{i_1}$ and $e_{i_2}$, and we denote the two outgoing edges by $e_{i_3}$ and $e_{i_4}$. We assume that at a fixed vertex $v_i$ the Euler circuit $\gamma$ of $G$ traverses successively the edges $e_{i_1}$ and $e_{i_3}$ in the first visit, and then it traverses successively $e_{i_2}$ and $e_{i_4}$ in the second visit. Merge the edges $e_{i_1}$ and $e_{i_3}$, i.e., replace them with a single new edge going from the tail of $e_{i_1}$ to the head of $e_{i_3}$. Also, we merge the edges $e_{i_2}$ and $e_{i_4}$ using the same technique and we delete the vertex $v_i$ to get a new graph with one less vertex. In the vicinity of the deleted vertex $v_i$, put a marker connecting the traces of the deleted vertex on each one of the new edges. We label each marker by $A$, $a$, $B$, or $b$ according to the coloring, the sign of the vertex and the order of the edges when traversing the path as explained in Figure \ref{fig4}. Such labeled marker for the vertex $v_{i}$ is denoted by  ``$i$-th marker''. We iterate this merging process until all the vertices are deleted and we obtain a single immersed circle in the plane with $n$ labeled markers in the vicinities of the deleted vertices. We call such data the $\gamma$-\textit{state} of the corresponding Euler circuit $\gamma$ of $G$. In Figure \ref{fig6}, we depict the $\gamma$-state of the Euler circuit $\gamma$ given in Figure \ref{fig3}.
\begin{figure}[H]
\centering
\includegraphics[scale=0.18]{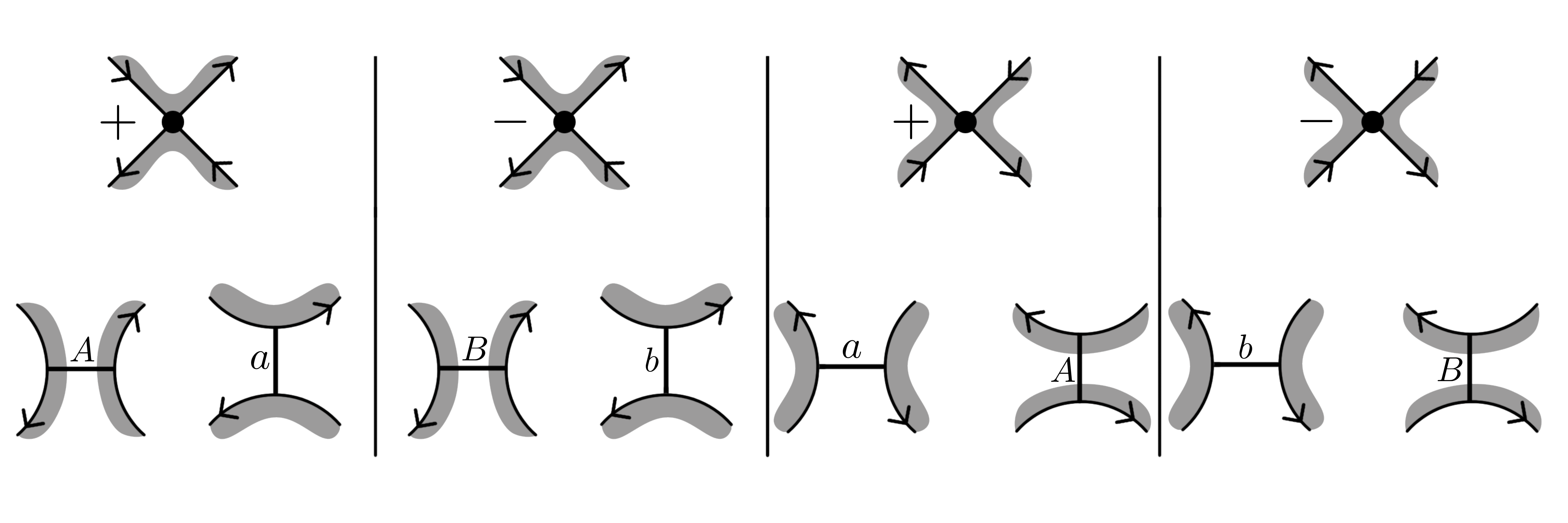}   
\caption{The $i$-th marker corresponding to the deleted vertex $v_i$ and the rule of its labeling.}
\label{fig4}
\end{figure}

\begin{figure}[H]
\centering
\includegraphics[scale=0.2]{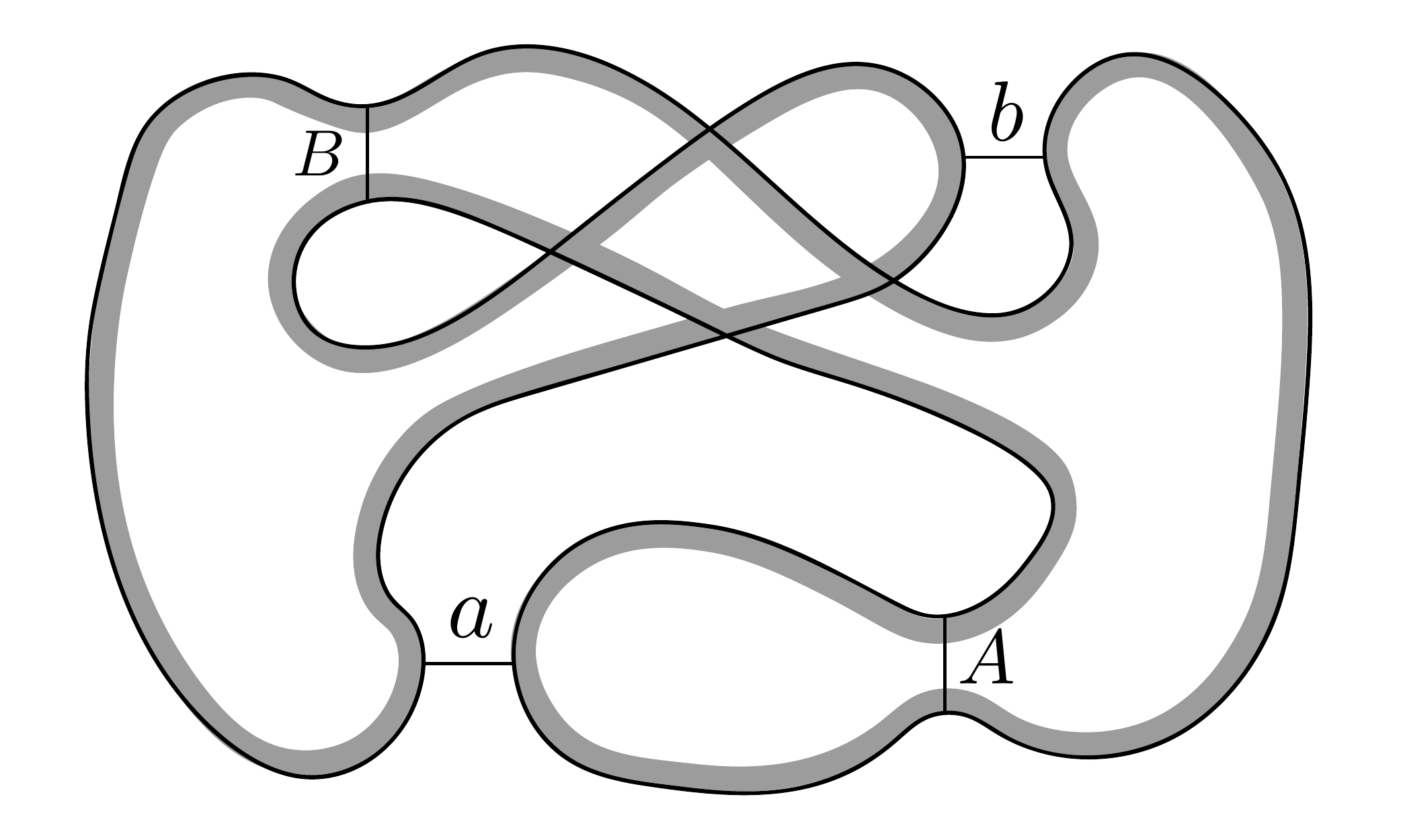}   
\caption{The $\gamma$-state of the Euler circuit $\gamma$ given in Figure \ref{fig3}.}
\label{fig6}
\end{figure}
For a fixed coloring of the graph $G$, a vertex $v_i$ of $G$ is \textit{internal} to $\gamma$ if the $i$-th marker is either an $A$-marker or $B$-marker. Otherwise, $v_i$ is \textit{external} to $\gamma$. The vertex $v_i$ is \textit{live} with respect to $\gamma$ if $\mathcal{C}_i(\gamma) \subset \left\lbrace i+1, \hdots ,n \right\rbrace$. Otherwise, $v_i$ is \textit{dead} with respect to $\gamma$. Being internal or external and live or dead determine the \textit{state} of the vertex $v_i$ with respect to $\gamma$.
The \textit{activity word} $w(\gamma)$ of the Euler circuit $\gamma$ is a word of $n$ letters in the alphabet $\left\lbrace L,D,l,d,\overline{L},\overline{D},\overline{l},\overline{d} \right\rbrace$ such that the $i$-th \textit{activity letter} $w_i(\gamma)$ corresponding to the vertex $v_i$ is chosen according to Table \ref{table1}. 

\begin{table}[ht]
\begin{center}
\begin{tabular}{|ll|l|l|}
\hline
\multicolumn{2}{|l|}{State of $v_i$ with respect to $\gamma$}                   & The label of the $i$-th marker & The activity letter $w_i(\gamma)$ \\ \hline
\multicolumn{1}{|l|}{\multirow{4}{*}{Internal}} & \multirow{2}{*}{Live} & $A$   &  {$L$} \\  \cline{3-4} 
\multicolumn{1}{|l|}{}                          &                       & $B$   & {$\overline{L}$}    \\   \cline{2-4} 
\multicolumn{1}{|l|}{}                          & \multirow{2}{*}{Dead} & $A$   & {$D$}   \\   \cline{3-4} 
\multicolumn{1}{|l|}{}                          &                       & $B$   &  {$\overline{D}$}   \\  \hline
\multicolumn{1}{|l|}{\multirow{4}{*}{External}} & \multirow{2}{*}{Live} & $a$   & {$l$}    \\  \cline{3-4} 
\multicolumn{1}{|l|}{}                          &                       & $b$ &  {$\overline{l}$}    \\  \cline{2-4} 
\multicolumn{1}{|l|}{}                          & \multirow{2}{*}{Dead} & $a$ & {$d$}   \\ \cline{3-4} 
\multicolumn{1}{|l|}{}                          &                       & $b$ & {$\overline{d}$}   \\ \hline
\end{tabular}
\caption{The activity of the vertex $v_{i}$ with respect to the Euler circuit $\gamma$ and the corresponding letter $w_i(\gamma)$.}
\label{table1}
\end{center}
\end{table}

Once the activity word $w(\gamma)$ is established, we assign to each letter $w_i(\gamma)$ a monomial called the \textit{weight} of the vertex $v_i$ with respect to $\gamma$ and denoted by $\mu_i(\gamma)\in \mathbb{Z}\left[ q^{-1},q \right]$ according to Table \ref{table2}. 

\begin{table}[H]
\centering
\begin{tabular}{|c|c|c|c|c|c|c|c|c|} \hline 
$w_i(\gamma)$ & $L$ & $D$ & $l$ & $d$ &  $\overline{L}$ & $\overline{D}$ &  $\overline{l}$ & $\overline{d}$ \\
\hline
$\mu_i(\gamma)$ & $-q^{-3}$ & $q$ & $-q^3$ & $q^{-1}$ & $-q^3$ & $q^{-1}$ & $-q^{-3}$ & $q$ \\
\hline
\end{tabular}
\caption{The weight of the letter $w_{i}(\gamma)$ in the word $w(\gamma)$.}
\label{table2}
\end{table}

The \textit{weight} of the Euler circuit $\gamma$ is the monomial $\mu(\gamma) := \displaystyle \prod_{1 \leq i \leq n} \mu_i(\gamma)$. Now we define the Laurent polynomial $X_G$ for $G$ as follows: \[ X_G(q) := \displaystyle \sum_{\text{Euler circuits } \gamma \text{ of } G}\mu(\gamma).\]
This definition of such a polynomial can be also extended to the case of disconnected $2$-digraphs as follows: 
\[ X_G(q) = (-(q^2+q^{-2}))^{m-1}\displaystyle \prod_{1 \leq i \leq m} X_{G_i}(q),\] where $G_1, \hdots , G_m$ are the connected components of the graph $G$.\\
 Now we state the main theorem of this paper whose proof will be given in the next section.
\begin{theorem}\label{main}
Let $G$ be a vertex-signed checkerboard-colorable $2$-digraph. The polynomial $X_G(q)$ is an invariant of the graph-isomorphism class of $G$ and it is independent from the chosen checkerboard coloring of $G$ and the labeling of the vertices. 
\label{theorem1}
\end{theorem}

\begin{proposition}
Let $\overline{G}$ be the graph obtained from the graph $G$ after changing all the signs of the 4-valent vertices, then we have $X_{\overline{G}}(q) = X_{G}(q^{-1})$.
\end{proposition}

\begin{proof}
This follows as a result of the fact that $\mu_{G}(\gamma)(q) = \mu_{\overline{G}}(\gamma)(q^{-1})$, where $\mu_{G}(\gamma)$ and $\mu_{\overline{G}}(\gamma)$ are the weights of the Euler circuit $\gamma$ of the graphs $G$ and $\overline{G}$, respectively.
\end{proof}

We enclose this subsection by stating the following theorem whose proof will be given in the next section which implies that the polynomial $X_G(q)$ satisfies a skein relation. For this end, we let $G^v$ be a checkerboard-colorable vertex-signed $2$-digraph with a fixed vertex $v$ depicted in Figure \ref{figure1}. The $2$-digraphs resulting from the two possible merging operations at $v$ described in Section 3 will be denoted by $G_0^{v}$ and $G_1^{v}$ respectively according to the scheme in Figure \ref{figure1}. 

\begin{figure}[h]
\centering
\includegraphics[scale=0.25]{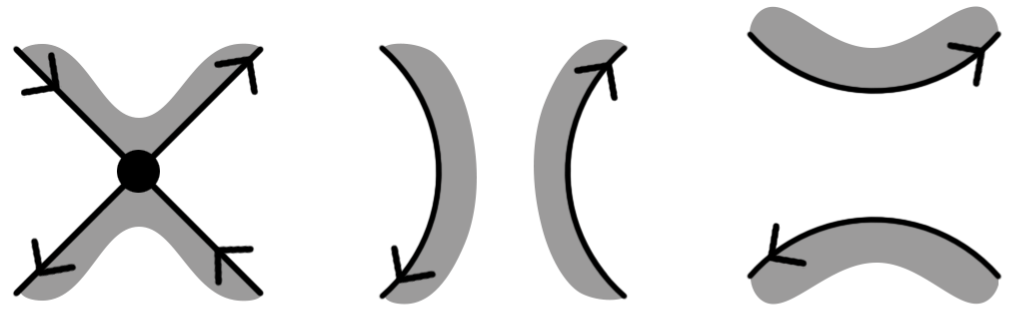}  
	\caption{The 2-digraph $G^v$ at the fixed vertex $v$ and the corresponding 2-digraphs $G_0^{v}$ and $G_1^{v}$ respectively obtained by the two merging operations.}
	\label{figure1}
\end{figure}

\begin{theorem}\label{theorem2}
Let $G$ be a checkerboard-colorable vertex-signed 2-digraph, then the polynomial $X_{G}$ satisfies the following skein relations:
\begin{enumerate}
\item If $v$ is equipped with a positive sign, then \[X_{G^{v}}(q) = q  X_{G_0^{v}}(q) + q^{-1}  X_{G_1^{v}}(q).\] 
\item  If $v$ is equipped with a negative sign, then \[X_{G^{v}}(q) = q^{-1}  X_{G_0^{v}}(q) + q  X_{G_1^{v}}(q).\] 
\end{enumerate}
\end{theorem}

As an application, we relate the polynomial introduced in this section to the established Jones-Kauffman polynomial of checkerboard-colorable virtual links. We start by recalling a brief background on virtual knot theory focusing on checkeroboard-colorable virtual links. Given a virtual link diagram, one can obtain a 4-valent virtual graph by replacing each classical crossing with a vertex and each virtual crossing by a virtual double point. The obtained graph is called the \textit{shadow graph} of the given virtual link diagram. If this graph is checkerboard-colorable, then the associated virtual link diagram is checkerboard-colorable. For instance, the virtual link diagram in Figure \ref{fig10}, whose shadow graph is exactly the graph depicted in Figure \ref{fig2}, is checkerboard-colorable.  A notable fact about the checkerboard-colorability is that it behaves well under generalized Reidemeister moves since any two equivalent checkerboard-colorable virtual link diagrams are related by a finite sequence of moves which preserve the checkerboard-colorability as shown in \cite[Theorem\,2.9]{Im}. A virtual link is \textit{checkerboard-colorable} if it admits a checkerboard-colorable diagram. We point out that the notion of checkerboard-colorable virtual links was first introduced by Kamada in \cite{K, K2}. 

The shadow graph of a checkerboard-colorable virtual link diagram can be made into a vertex-signed 4-valent graph following the rules explained in Figure \ref{fig9}. As an example, the vertex-signed graph in the left of Figure \ref{fig3} is exactly the shadow graph of the virtual link diagram depicted in Figure \ref{fig10}. 

\begin{figure}[H]
\centering
\includegraphics[scale=0.18]{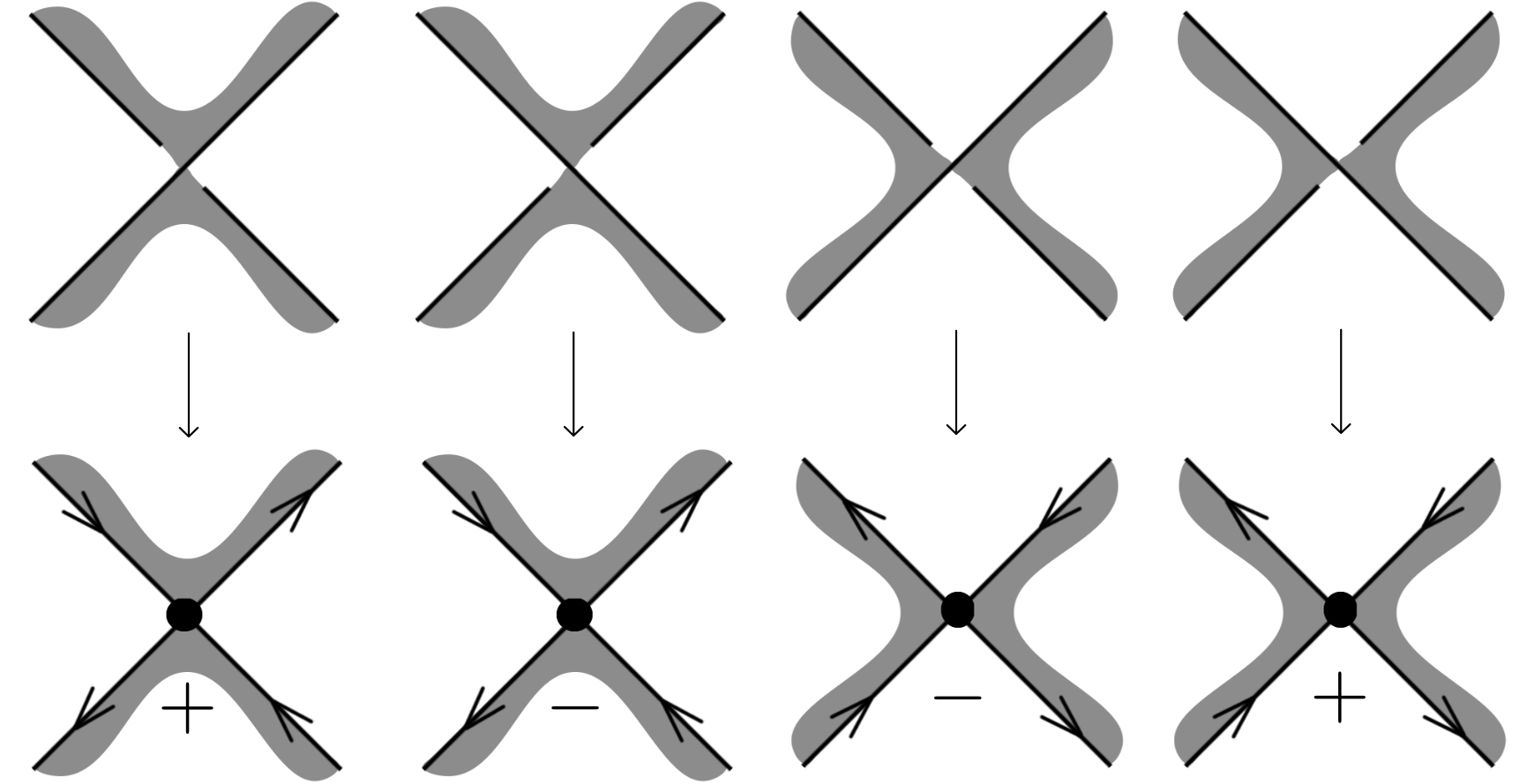}   
\caption{The sign assigning rule for the vertices of the shadow graph.}
\label{fig9}
\end{figure}

\begin{corollary} \label{corollary5.1}
Let $L$ be an oriented checkerboard-colorable virtual link. Let $G$ be the shadow graph of a checkerboard-colorable diagram of $L$ made into a vertex-signed $4$-valent graph, then $f_{L}(q) = (-q)^{-3\omega(L)} X_{G}(q)$, where $f_{L}(q)$ and $\omega(L)$ are the Jones-Kauffman polynomial and the writhe of the oriented virtual link $L$, respectively.
\end{corollary}

\begin{proof}
  Let $D$ be a checkerboard-colorable diagram of $L$. By Theorem \ref{theorem2}, the polynomial $X_G(q)$ is exactly the bracket polynomial of $D$. Then the result follows.
\end{proof}

\begin{example}
We consider the checkerboard colorable knot $K = 5.2426$ represented by the diagram on the left of Figure \ref{fig13}. We compute the Jones-Kauffman polynomial of $K$ using our method. The shadow graph of such a knot diagram is depicted on the right of Figure \ref{fig13}. 

\begin{figure}[H]
\centering
\includegraphics[scale=0.18]{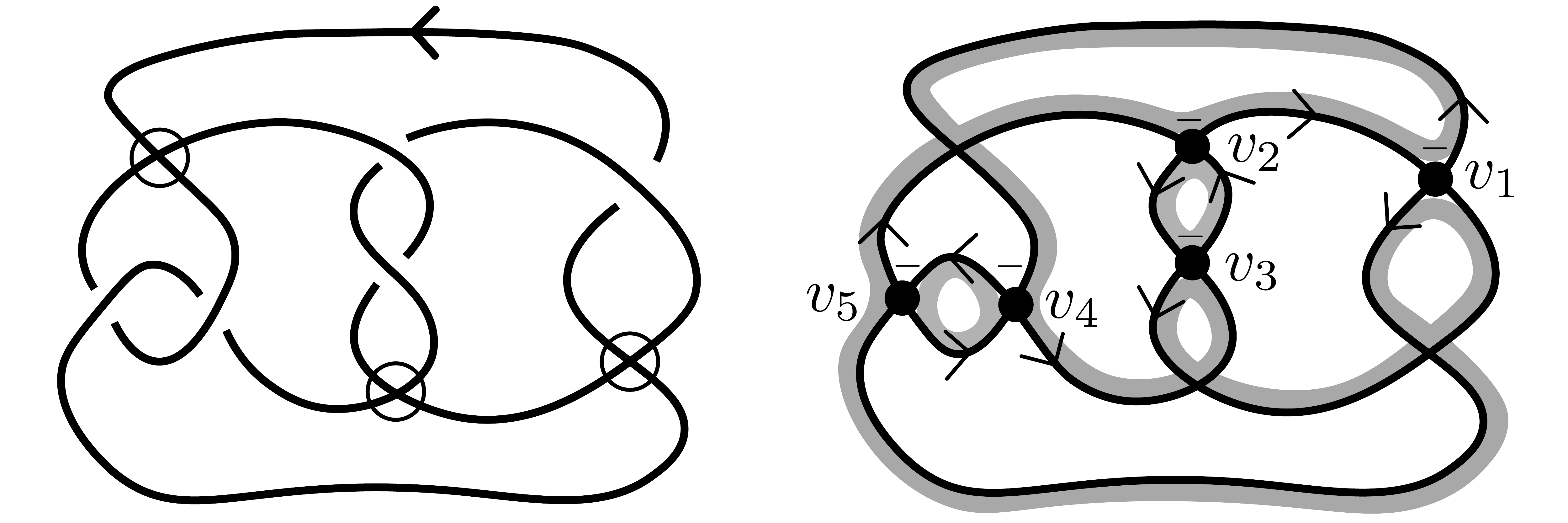}   
\caption{A diagram of the knot $K = 5.2426$ and its shadow graph, respectively.}
\label{fig13}
\end{figure}

The following are the chord diagrams of the nine Euler circuits of the shadow graph together with their activity words and weights. We have placed the label of the marker corresponding to each vertex as the exponent of the vertex label.  

\begin{figure}[H]
\centering
\includegraphics[scale=0.20]{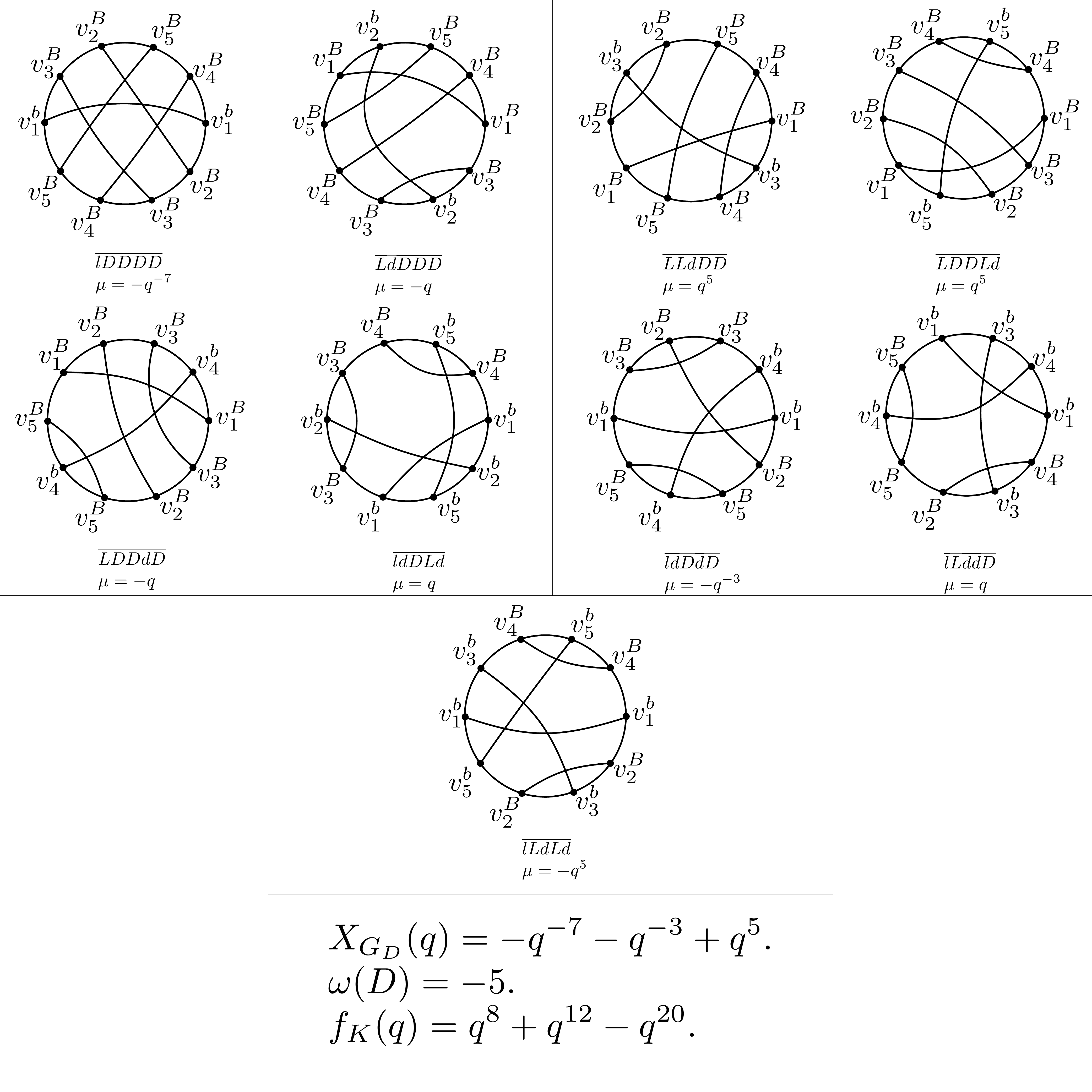}   
\label{fig14}
\end{figure}
\label{example1}
\end{example}

\section{Proof of the Main Theorems}
  We start by recalling some operators defined for graphs and Euler circuits used by Arratia et al in \cite{arratia}. Then we give some lemmas which will be substantial for the proof of the main theorem. For this end, we let $\gamma$ be an Euler circuit of a vertex-signed 2-digraph $G$ with vertices $\{v_1,\dots ,v_n\}$.

\begin{definition}
For any pair of distinct vertices $u$ and $v$ of $G$, we let the graph $G_{uv}$ (or the circuit $\gamma_{uv}$) denote the same object with the labels of $u$ and $v$ being swapped.
\end{definition}

\begin{remark}
If $v_i$ and $v_j$ are two distinct vertices of $G$, then it is clear that the label of the $i$-th marker (respectively $j$-th marker) in the $\gamma$-state is the same as the $j$-th marker (respectively $i$-th marker) in the $\gamma_{v_i v_j}$-state.
\label{remark1}
\end{remark}

\begin{definition}
The \textit{interlace graph} $H = H(\gamma)$ of the 2-digraph $G$ corresponding to the circuit $\gamma$ has the same vertex set as $G$, with an edge $uv$ in $H(\gamma)$ if $u$ and $v$ are interlaced in $\gamma$.
\end{definition}
\begin{definition} 
For any pair of interlacing vertices $u$ and $v$ in the Euler circuit $\gamma$, we let $\gamma^{uv}$ be the circuit obtained by exchanging one of the edge sequences from $u$ to $v$ with the other. This circuit is called the transposition circuit obtained by transposing the circuit $\gamma$ on the pair $uv$. Such operation is well-defined since the circuit $\gamma$ with two interlacing vertices $u$ and $v$ can be written as $\ldots u \ldots v \ldots u \ldots v \ldots$.
\end{definition}
As an example, in Figure \ref{fig3} $\gamma$ can be written as $v_{1}e_{1}v_{2}e_{2}v_{4}e_{3}v_{1}e_{4}v_{2}e_{5}v_{3}e_{6}v_{4}e_{7}v_{3}e_{8}v_{1}$ while $\gamma^{v_{3}v_{4}}$ can be written as $v_{1}e_{1}v_{2}e_{2}v_{4}e_{7}v_{3}e_{6}v_{4}e_{3}v_{1}e_{4}v_{2}e_{5}v_{3}e_{8}v_{1}$. The $\gamma$-state is depicted in Figure \ref{fig6} while the $\gamma^{v_{3}v_{4}}$-state is depicted in Figure \ref{fig7}. One can observe that the $3$rd marker and the $4$th marker are the only ones which change labels when one compares the two states.

\begin{figure}[H]
\centering
\includegraphics[scale=0.22]{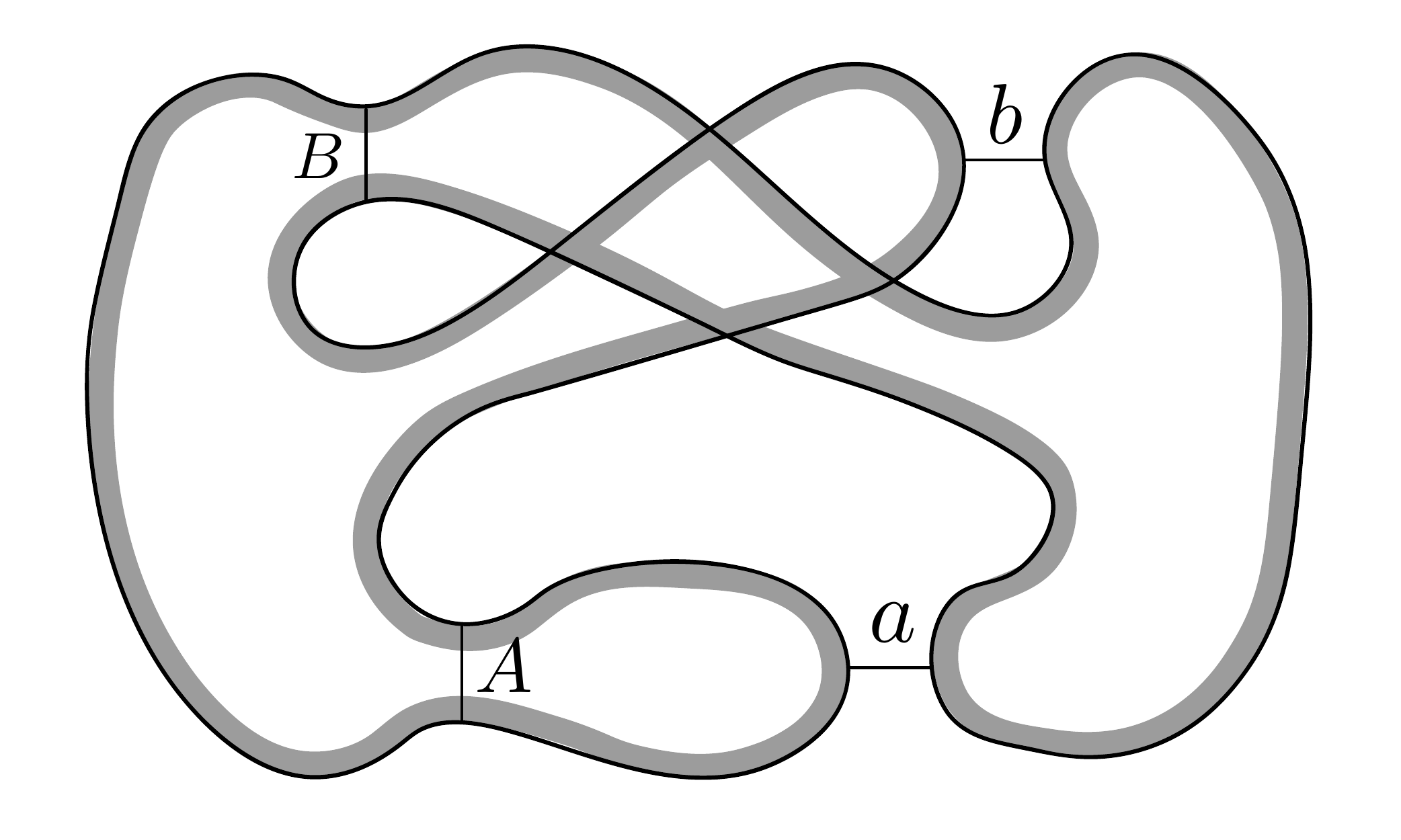}   
\caption{The $\gamma^{v_{3}v_{4}}$-state obtained from the Euler circuit $\gamma$ given in Figure \ref{fig3}.}
\label{fig7}
\end{figure}
\begin{remark}
If $v_i$ and $v_j$ interlace in $\gamma$ and $k \neq i,j$, then the $k$-th marker in the $\gamma$-state has the same label as the $k$-th marker in the $\gamma^{v_i v_j}$-state. On the other hand, the change of the labels of the $i$-th marker and the $j$-th marker in the $\gamma$-state compared to the $\gamma^{v_i v_j}$ can be described as follows: $A \leftrightarrow a$, $B \leftrightarrow b$. 
\label{remark2}
\end{remark}

According to \cite[Lemma\,4]{arratia}, any circuit of $G$ can be obtained from the circuit $\gamma$ by transposing on some interlaced pairs of vertices.

To define another graph operator, we choose a pair of distinct vertices $u$ and $v$ of $G$. Arratia et al. in \cite{arratia} partition the vertices of $G$ other than $u$ and $v$ into four classes:
  \begin{enumerate}
  \item vertices adjacent to $u$ alone;
  \item vertices adjacent to $v$ alone;
  \item vertices adjacent to both $u$ and $v$; and
  \item vertices adjacent to neither $u$ nor $v$.
  \end{enumerate}
\begin{definition}
The {\it pivot} graph by the fixed pair of distinct vertices $u$ and $v$ of $G$ denoted by $G^{uv}$ is obtained by toggling all pairs $xy$ of vertices starting with graph $G$ where $x$ is one of the classes (1-3) and $y$ is in a different class (1-3). The pair $xy$ is ``toggled'' means make it an edge if it was a non-edge and make it a non-edge if it was an edge. 
\end{definition}
\begin{remark}\label{basic}
We point out that an edge in a graph $G$ connecting $u$ or $v$ stays an edge in the pivot graph $G^{uv}$ and vice-versa.
\end{remark}
Now we quote the following key lemma whose proof can be found in \cite{arratia}.
\begin{lemma}\cite[Lemma\,7]{arratia}\label{lemma}
If $H=H(\gamma)$, and $H$ has an edge $uv$ (i.e., $u$ and $v$ are interlaced in $\gamma$), then $H^{uv}=H(\gamma^{uv})_{uv}$.
\end{lemma}
From now on, we denote $H$ to be the interlace graph $H(\gamma)$. Now we can state and prove the following lemmas which will be used in the proof of Theorem \ref{main}.

\begin{lemma}\label{lemma1}
Let $\gamma$ be an Euler circuit of a $2$-digraph with a pair of distinct vertices $v_i$ and $v_j$, then
\begin{enumerate}
\item If $v_k$ is a vertex other than $v_i$ and $v_j$, then $\mathcal{C}_k(\gamma) \setminus \left\lbrace i,j \right\rbrace=\mathcal{C}_k(\gamma_{v_i v_j}) \setminus \left\lbrace i,j \right\rbrace$.
\item If $v_i$ and $v_j$ don not interlace in $\gamma$, then $\mathcal{C}_i(\gamma)=\mathcal{C}_j(\gamma_{v_i v_j})$ and $ \mathcal{C}_j(\gamma)=\mathcal{C}_i(\gamma_{v_i v_j})$.
\item If $v_i$ and $v_j$ interlace in $\gamma$, then $i \in \mathcal{C}_j(\gamma_{v_i v_j}), j \in \mathcal{C}_i(\gamma_{v_i v_j}), \mathcal{C}_i(\gamma)\setminus \left\lbrace j \right\rbrace  = \mathcal{C}_j(\gamma_{v_i v_j})\setminus \left\lbrace i \right\rbrace, \mathcal{C}_j(\gamma)\setminus \left\lbrace i \right\rbrace  = \mathcal{C}_i(\gamma_{v_i v_j})\setminus \left\lbrace j \right\rbrace$.
\end{enumerate}
\end{lemma}
\begin{proof}
  
    \begin{enumerate}
  \item The result follows from the fact that the vertices $v_k$ and $v_h$ interlaces in $\gamma$ iff they interlace in $\gamma_{v_{i}v_{j}}$ for $k,h \neq i,j$.
  \item If $v_i $ and $v_j$ don not interlace in $\gamma$, then we have
    \begin{align*}
    h\in C_i(\gamma) \  \Leftrightarrow &  \ v_h\ \text{interlaces with} \ v_i \ \text{in} \ \gamma\\ \Leftrightarrow & \ v_h \ \text{interlaces with} \ v_j \ \text{in} \ \gamma_{v_i v_j}\\ \Leftrightarrow & \ h\in \mathcal{C}_j(\gamma_{v_iv_j})
    \end{align*}
  \item If $v_i $ and $v_j$ interlace in $\gamma$, then after swaping the labels of the vertices $v_i$ and $v_j$the vertices $v_j$ and $v_i$ interlace in $\gamma_{v_i v_j}$. This implies that  $i\in\mathcal{C}_j(\gamma_{v_i v_j})$ and $j\in\mathcal{C}_i(\gamma_{v_i v_j})$. Also, we have
    \begin{align*}
      h\in\mathcal{C}_i(\gamma)\setminus\{j\}\ \Leftrightarrow & \ v_h\ \ \text{interlaces with}\ \ v_i\ \ \text{in}\ \ \gamma\ \ \text{and}\ \ v_h\ne v_j\\\Leftrightarrow & \ v_h\ \ \text{interlaces with}\ \ v_j\ \ \text{in}\ \ \gamma_{v_i v_j}\ \ \text{and}\ \ v_h\ne v_i\\ \Leftrightarrow &\ \  h\in\mathcal{C}_j(\gamma)\setminus\{i\}.
    \end{align*}
  \end{enumerate}
\end{proof}

\begin{lemma}\label{lemma2}
If $\gamma$ is an Euler circuit of a $2$-digraph with an interlacing pair of vertices $v_i$ and $v_j$, then
\begin{enumerate}
\item  $i \in \mathcal{C}_j(\gamma^{v_i v_j}),  j \in \mathcal{C}_i(\gamma^{v_i v_j})$.
\item $\mathcal{C}_i(\gamma) \setminus \left\lbrace j \right\rbrace = \mathcal{C}_j(\gamma^{v_i v_j}) \setminus \left\lbrace i \right\rbrace,  \mathcal{C}_j(\gamma)\setminus \left\lbrace i \right\rbrace = \mathcal{C}_i(\gamma^{v_i v_j})\setminus \left\lbrace j \right\rbrace$.
\item $\mathcal{C}_i(\gamma)=\mathcal{C}_i((\gamma^{v_i v_j})_{v_i v_j}),  \mathcal{C}_j(\gamma)=\mathcal{C}_j((\gamma^{v_i v_j})_{v_i v_j})$.
\end{enumerate}
\end{lemma}

\begin{proof} 
 \begin{enumerate}
\item  We note that when we transpose the circuit $\gamma$ on the pair $v_i$ and $v_j$, the vertices $v_i$ and $v_j$ stay in place. This implies that $v_i$ and $v_j$ interlace in $\gamma^{v_i v_j}$, and so $i \in \mathcal{C}_j(\gamma^{v_i v_j})$ and $j \in \mathcal{C}_i(\gamma^{v_i v_j})$.
\item We have 
\begin{align*}
k\in\mathcal{C}_j(\gamma^{v_i v_j}) \setminus \left\lbrace i \right\rbrace \Leftrightarrow & \  v_k \ \text{interlaces with} \ v_j \ \text{in} \ \gamma^{v_i v_j}\ \text{and}\ k\ne i\\
 \Leftrightarrow & \ \text{there is an edge} \ e \ \text{connecting} \ v_k \ \text{to}\ v_j\ \text{in}\ H(\gamma^{v_i v_j})\ \text{and}\ k\ne i\\
 \Leftrightarrow & \ \text{there is an edge}\ e\ \text{connecting} \ v_k\ \text{to}\ v_i\ \text{in}\ H(\gamma^{v_i v_j})_{v_iv_j}\ \text{and}\ k\ne j\\
\Leftrightarrow & \  \text{there is an edge $e$ in} \ H^{v_iv_j} \ \text{connecting} \ v_k \ \text{to}\  v_i\ \text{and}\ k\ne j\\
\Leftrightarrow & \ \text{there is an edge $e$  in} \ H \ \text{connecting} \ v_k\ \text{to}\ v_i\ \text{and}\ k\ne j\\
\Leftrightarrow & \ v_k\ \text{interlaces}\ v_i\ \text{in}\ \gamma\ \text{and}\ k\ne j\\
 \Leftrightarrow & \ k\in C_i(\gamma)\setminus\{j\},
\end{align*}
where the fourth equivalence follows from the fact $H^{v_iv_j}=H(\gamma^{v_i v_j})_{v_iv_j}$ in Lemma \ref{lemma}.

The result $\mathcal{C}_j(\gamma)\setminus \left\lbrace i \right\rbrace = \mathcal{C}_i(\gamma^{v_i v_j})\setminus \left\lbrace j \right\rbrace$ can be obtained by a similar manner after interchanging $i$ and $j$ in the previous argument.
  \item We have
\begin{align*}
 k\in\mathcal{C}_i((\gamma^{v_i v_j})_{v_iv_j}) \Leftrightarrow & \ v_k\ \text{interlaces}\ v_i\ \text{in}\ (\gamma^{v_i v_j})_{v_iv_j}\\ \Leftrightarrow & \ v_k\ \text{interlaces}\ v_j\ \text{in}\ \gamma^{v_i v_j}\\
\Leftrightarrow & \ \text{there is an edge} \ e\ \text{connecting} \ v_k\ \text{to}\ v_j\ \text{in}\ H(\gamma^{v_i v_j})=(H^{v_iv_j}(\gamma))_{v_iv_j}\\
\Leftrightarrow & \ \text{there is an edge} \ e\ \text{connecting} \ v_k\ \text{to}\ v_i\ \text{in}\ H^{v_iv_j} \\ \Leftrightarrow & \ v_k\ \text{interlaces with}\ v_i\ \text{in}\ \gamma\\  \Leftrightarrow & \ k\in C_i(\gamma)\\
\end{align*}
\end{enumerate}
\end{proof}

For the next lemma, we partition the set $\left\lbrace 1, \hdots , n \right\rbrace \setminus \left\lbrace i,j \right\rbrace$ into the following disjoint sets: $A_{ij}=\overline{\mathcal{C}}_i(\gamma)  \cap  \overline{\mathcal{C}}_j(\gamma), B_{ij}= (C_i(\gamma)\setminus\{j\}) \cap \overline{C_j}(\gamma), C_{ij}=(C_j(\gamma)\setminus\{i\}) \cap \overline{C_i}(\gamma), D_{ij}=\mathcal{C}_i(\gamma) \cap \mathcal{C}_j(\gamma)$.
\begin{lemma}\label{lemma3}
Let $\gamma$ be an Euler circuit of a $2$-digraph with an interlacing pair of vertices $v_i$ and $v_j$, then

\begin{enumerate}
\item  If $k \in A_{ij}$, then $\mathcal{C}_k((\gamma^{v_i v_j})_{v_i v_j})=\mathcal{C}_k(\gamma)$.
\item If $k \in B_{ij}$, then 

$\mathcal{C}_k((\gamma^{v_i v_j})_{v_i v_j})=\left\lbrace i \right\rbrace \cup \left( \mathcal{C}_k(\gamma) \cap (A_{ij} \cup B_{ij}) \right) \cup \left( \overline{\mathcal{C}}_k(\gamma)\setminus \left\lbrace k \right\rbrace \cap (C_{ij} \cup D_{ij}) \right)$.
\item If $k \in C_{ij}$, then 

$\mathcal{C}_k((\gamma^{v_i v_j})_{v_i v_j})=\left\lbrace j \right\rbrace \cup \left( \mathcal{C}_k(\gamma) \cap (A_{ij} \cup C_{ij}) \right) \cup \left( \overline{\mathcal{C}}_k(\gamma)\setminus \left\lbrace k \right\rbrace \cap (B_{ij} \cup D_{ij}) \right)$.
\item If $k \in D_{ij}$, then 

$\mathcal{C}_k((\gamma^{v_i v_j})_{v_i v_j})=\left\lbrace i,j \right\rbrace \cup \left( \mathcal{C}_k(\gamma) \cap (A_{ij} \cup D_{ij}) \right) \cup \left( \overline{\mathcal{C}}_k(\gamma)\setminus \left\lbrace k \right\rbrace \cap (B_{ij} \cup C_{ij}) \right)$.
\end{enumerate} 
\end{lemma}
\begin{proof}
\begin{enumerate}
  \item If $k\in A_{ij}$, then we have 
\begin{align*}
  h\in C_k((\gamma^{v_i v_j})_{v_iv_j}) \Leftrightarrow& \ v_h\ \text{interlaces with}\ v_k\ \text{in}\ (\gamma^{v_i v_j})_{v_iv_j}\\
 \Leftrightarrow & \ \text{there is an edge} \ e \ \text{connecting}\ v_h\ \text{to}\ v_k\ \text{in}\ H(\gamma^{v_i v_j})_{v_iv_j}\\
  \Leftrightarrow &  \ \text{there is an edge} \ e \ \text{connects} \ v_h\ \text{to}\ v_k\ \text{in} \\ & H^{v_i v_j} = H(\gamma^{v_i v_j})_{v_iv_j}\ \text{and}   \ v_k\ \text{not connected to}\ v_i\ \text{nor}\ v_j\\
 \Leftrightarrow & \ \text{there is an edge} \ e\ \text{connecting}\ v_h\ \text{to}\ v_k\ \text{in}\ H\\
 \Leftrightarrow & \ h\in C_k(\gamma).
\end{align*}

\item If $k\in B_{ij}$, then $k\not\in A_{ij}$ or simply $k\in C_i(\gamma)\cup C_j(\gamma)$ with $k\ne i$ and $k\ne j$. Now 
\begin{align*}
h\in C_k((\gamma^{v_i v_j})_{v_iv_j}) \Leftrightarrow & \ v_h \ \text{interlaces with} \ v_k \ \text{in} \ (\gamma^{v_i v_j})_{v_iv_j}\\   \Leftrightarrow & \ v_h \ \text{is connected to} \ v_k \ \text{ by an edge in} \ H((\gamma^{v_i v_j})_{v_iv_j})=H^{v_iv_j}\\ \Leftrightarrow & \ h=i \ \text{or}  \ (h\in C_k(\gamma) \ \text{and}  \ (h\in B_{ij} \  \text{or} \ h\in A_{ij}))\\   & \  \text{or} \ (h\in \overline{C}_k(\gamma)\setminus\{k\} \ \text{and} \ (h\in C_{ij} \ \text{or} \ h\in D_{ij})).
\end{align*}
The third equivalence is true since either $v_h=v_i$ or the edges of $H^{v_iv_j}$ connecting $v_h$ to $v_k$ are:

\begin{itemize}
  \item The edges of $H$ for which one of the ends is $v_k$ and the other end $v_h$ is connected to $v_i$ alone or connected neither to $v_i$ nor $v_j$,
  \item The edges obtained by connecting the vertices $v_h$ which are not connected to $v_k$ in $H$ and which are connected to $v_j$ alone or to both $v_i$ and $v_j$.
  \end{itemize}

\item This can be proved by a similar argument as in (2) after interchanging the roles of $i$ and $j$.

\item We have $k\in D_{ij}=\mathcal{C}_i(\gamma) \cap \mathcal{C}_j(\gamma)$ which implies that $v_k$ interlaces with both $v_i$ and $v_j$ but  $v_k\ne v_i$ and $v_k\ne v_j$.
  \begin{align*}
    h\in C_k((\gamma^{v_i v_j})_{v_iv_j}) \Leftrightarrow & \ v_h \ \text{interlaces with} \ v_k \ \text{in} \ (\gamma^{v_i v_j})_{v_iv_j}\\   \Leftrightarrow & \ v_h \ \text{is connected to} \ v_k \ \text{ by an edge in} \ H((\gamma^{v_i v_j})_{v_iv_j})=H^{v_iv_j}
  \end{align*}
  Now the construction of the pivot graph shows that any edge of $H^{v_iv_j}$ whose one end is $v_k$ comes from the interlace graph $H$ of $\gamma$ in three ways:

 \begin{itemize}
  \item $v_h$ is $v_i$ or $v_j$, i.e., $h\in\{i,j\}$ or
  \item $v_h$ is neither $v_i$ nor $v_j$ and $v_h$ is connected to $v_k$ in $H$ and (($v_h$ is connected to $v_i$ and $v_j$ in $H$) or ($v_h$ is neither connected to $v_i$ nor $v_j$ in $H$)), i.e., $h\in \left( \mathcal{C}_k(\gamma) \cap (A_{ij} \cup D_{ij}) \right)$.
 \item $v_h$ is neither $v_i$ nor $v_j$ and $v_h$ is not connected to $v_k$ in $H$ and (($v_h$ is connected to $v_i$ alone) or ($v_h$ is connected to $v_j$ alone nor $v_j$)), i.e., $\left( \overline{\mathcal{C}}_k(\gamma)\setminus \left\lbrace k \right\rbrace \cap (B_{ij} \cup C_{ij}) \right)$.
  \end{itemize}
  The two last assertions correspond to the toggling operations.
\end{enumerate}

\end{proof}

\begin{lemma}\label{lemma4}
Let $\gamma$ be an Euler circuit of a $2$-digraph $G$ and $v_i$ be a fixed vertex. If $v_k$ is a vertex other than $v_i$ and $v_{i+1}$, then 
\begin{enumerate}
\item $w_k(\gamma)=w_k(\gamma_{v_i v_{i+1}})$. 

\item If $v_i$ and $v_{i+1}$ interlace in $\gamma$, then $w_k(\gamma)=w_k((\gamma^{v_i v_{i+1}})_{v_i v_{i+1}})$ in the following cases:
\begin{enumerate}
\item $k\in A_{i\,i+1}$
\item $i < k$.
\item $v_i$ is live with respect to $\gamma$.
\item $v_i$ is dead with respect to $\gamma$ and $\mathcal{C}_{i+1}(\gamma) \setminus \left\lbrace i \right\rbrace \subset \left\lbrace i+2, \hdots , n \right\rbrace$.
\end{enumerate}
\end{enumerate}

\end{lemma}

\begin{proof}
\begin{enumerate}
\item This follows directly from Lemma \ref{lemma1}(1). 
\item The label of the $k$-th marker in the $\gamma$-state is the same as the $k$-th marker in the $(\gamma^{v_i v_{i+1}})_{v_i v_{i+1}}$-state, then we just need to show that $v_k$ has the same activity (live or dead) with respect to $\gamma$ and $(\gamma^{v_i v_{i+1}})_{v_i v_{i+1}}$.   
\begin{enumerate}
\item This case follows by Lemma \ref{lemma3}(1).
\item We can assume that $k \notin A_{i i+1}$ which implies that $k \in \mathcal{C}_i(\gamma) \cup \mathcal{C}_{i+1}(\gamma)$. Now since $i < k$ and $k \neq i+1$, then $i+1<k$. So $v_k$ is dead with respect to $\gamma$. By Lemma \ref{lemma3}, the set $\mathcal{C}_k((\gamma^{v_i v_{i+1}})_{v_i v_{i+1}})$ contains either $i$ or $i+1$. Thus, $v_k$ is dead also with respect to $(\gamma^{v_i v_{i+1}})_{v_i v_{i+1}}$.
\item We can assume that $k \notin A_{i i+1}$, $k < i$, and $v_i$ is live with respect to $\gamma$. This implies that $k \in C_{i i+1}$.  Assume that $v_k$ is live with respect to $\gamma$ and suppose on the contrary that $v_k$ is dead with respect to $(\gamma^{v_i v_{i+1}})_{v_i v_{i+1}}$ which implies that $\mathcal{C}_k(\gamma) \subset \left\lbrace k+1 , \hdots , n \right\rbrace$ and there exists $l \in \mathcal{C}_k((\gamma^{v_i v_{i+1}})_{v_i v_{i+1}})$ such that $l < k$. By Lemma \ref{lemma3}, we have $l \in \left\lbrace i+1 \right\rbrace \cup ( \overline{\mathcal{C}_k}(\gamma) \setminus \left\lbrace k \right\rbrace \cap (B_{i i+1} \cup D_{i i+1})) \subset \mathcal{C}_{i}(\gamma)$. Since $v_i$ is live in $\gamma$, then $i < l$, and we have $l<k$, then $i<k$ which is a contradiction so $k < l$, which is a contradiction. Hence, $v_k$ is live with respect to $(\gamma^{v_i v_{i+1}})_{v_i v_{i+1}}$ also.  Now we assume that $v_k$ is dead with respect to $\gamma$ which implies that there exists $l \in \mathcal{C}_k(\gamma)$ such that $l < k$. Since $B_{i i+1} \cup D_{i i+1} \subset \mathcal{C}_i(\gamma)$, then $l \in A_{i i+1} \cup C_{i i+1}$. By Lemma \ref{lemma3}, we have $l \in \mathcal{C}_k((\gamma^{v_i v_{i+1}})_{v_i v_{i+1}})$, and so $v_k$ is also dead with respect to $(\gamma^{v_i v_{i+1}})_{v_i v_{i+1}}$. 
\item Lemma \ref{lemma1}(3) and the assumption in this case imply that $v_i$ is live with respect to $\gamma_{v_i v_{i+1}}$, so the result follows from $(c)$ and $(1)$. 
\end{enumerate}
\end{enumerate}
\end{proof}
\begin{proof}[Proof of Theorem \ref{theorem1}]
Let $G$ be a checkerboard-colorable vertex-signed $2$-digraph. We start by showing that the two possible checkerboard colorings of $G$ yield the same polynomial. Specifically, as illustrated in Figure \ref{fig4} and summarized in Table \ref{table1}, the changes in the activity word of an Euler circuit due to changing the checkerboard coloring can be described as follows: $L \leftrightarrow \overline{l}, l \leftrightarrow \overline{L}, D \leftrightarrow \overline{d}, \text{and} \ d \leftrightarrow \overline{D}$. Consequently, according to Table \ref{table2}, the weight of the Euler circuit remains unchanged.\\
Now, we show the independence of the polynomial from the vertex labeling. To do so, it suffices to show that $X_G(q)=X_{G_{v_i v_{i+1}}}(q)$ for every integer $i$, $1 \leq i \leq n$. The argument of the proof consists of showing that for each Euler circuit of $G$, there exists a unique Euler circuit of $G_{v_i v_{i+1}}$ with the same weight. For this end, we consider all possible cases as follows:
\begin{enumerate}
\item  If $v_i$ and $v_{i+1}$ do not interlace in $\gamma$, then Lemma \ref{lemma1}(2) and Lemma \ref{lemma4}(1) imply that $w_i(\gamma) = w_{i+1}(\gamma_{v_i v_{i+1}})$, $w_{i+1}(\gamma) = w_i(\gamma_{v_i v_{i+1}})$, and $w_k(\gamma) = w_k(\gamma_{v_i v_{i+1}})$ for all $k \neq i,i+1$. Thus, the weights of $\gamma$ and $\gamma_{v_i v_{i+1}}$ are the same. 

\item If $v_i$ and $v_{i+1}$ interlace in $\gamma$, then we consider all the possible four sub-cases as follows:
\begin{enumerate}
\item If $v_i$ is live with respect to $\gamma$ and $\mathcal{C}_{i+1}(\gamma) \setminus \left\lbrace i \right\rbrace \subset \left\lbrace i+2, \hdots , n \right\rbrace$, then Lemma \ref{lemma4} implies that $w_k(\gamma) = w_k(\gamma_{v_i v_{i+1}}) = w_k(\gamma^{v_i v_{i+1}}) = w_k((\gamma^{v_i v_{i+1}})_{v_i v_{i+1}})$ for all $k \neq i,i+1$. Now according to Lemma \ref{lemma1}(3) and Lemma \ref{lemma2}(2),(3), we obtain that $v_i$ is live  and $v_{i+1}$ is dead with respect to $\gamma$, $\gamma_{v_i v_{i+1}}$, $\gamma^{v_i v_{i+1}}$, and $(\gamma^{v_i v_{i+1}})_{v_i v_{i+1}}$. Using Remark \ref{remark1} and Remark \ref{remark2}, we state the $i$-th and $(i+1)$-th letters of the activity words of the four circuits $\gamma$, $\gamma_{v_i v_{i+1}}$, $\gamma^{v_i v_{i+1}}$, and $(\gamma^{v_i v_{i+1}})_{v_i v_{i+1}}$ in all possible cases in Table \ref{table3}. We have either
\[
\begin{cases} 
    \mu_i(\gamma) \cdot \mu_{i+1}(\gamma) = \mu_i(\gamma_{v_i v_{i+1}}) \cdot \mu_{i+1}(\gamma_{v_i v_{i+1}}) ,& \\
    \mu_i(\gamma^{v_i v_{i+1}}) \cdot \mu_{i+1}(\gamma^{v_i v_{i+1}}) = \mu_i((\gamma^{v_i v_{i+1}})_{v_i v_{i+1}}) \cdot \mu_{i+1}((\gamma^{v_i v_{i+1}})_{v_i v_{i+1}}).
\end{cases}
\]
\[
\quad \text{or} \quad
\begin{cases} 
    \mu_i(\gamma) \cdot \mu_{i+1}(\gamma) =  \mu_i((\gamma^{v_i v_{i+1}})_{v_i v_{i+1}}) \cdot \mu_{i+1}((\gamma^{v_i v_{i+1}})_{v_i v_{i+1}}),& \\
    \mu_i(\gamma^{v_i v_{i+1}}) \cdot \mu_{i+1}(\gamma^{v_i v_{i+1}}) =  \mu_i(\gamma_{v_i v_{i+1}}) \cdot \mu_{i+1}(\gamma_{v_i v_{i+1}}).
\end{cases}
\]

Finally, each Euler circuit of $G$ has exactly the same weight as a unique Euler circuit of $G_{v_i v_{i+1}}$.

\begin{table}[h]
  \centering
  \scriptsize{
\begin{tabular}{|cc|cc|cc|cc|}
  \hline
\multicolumn{2}{|c|}{$\gamma$} & \multicolumn{2}{c|}{$\gamma^{v_i v_{i+1}}$} & \multicolumn{2}{c|}{$\gamma_{v_i v_{i+1}}$} & \multicolumn{2}{c|}{$(\gamma^{v_i v_{i+1}})_{v_i v_{i+1}}$} \\ 
$w_i$ & $w_{i+1}$ & $w_i$ & $w_{i+1}$ & $w_i$ & $w_{i+1}$ & $w_i$ & $w_{i+1}$ \\ \hline
$L$ & $D$ & $l$ & $d$ & $L$ & $D$ & $l$ & $d$ \\
$L$ & $\overline{D}$ & $l$ & $\overline{d}$ & $\overline{L}$ & $D$ & $\overline{l}$ & $d$ \\ 
$L$ & $d$ & $l$ & $D$ & $l$ & $D$ & $L$ & $d$ \\
$L$ & $\overline{d}$ & $l$ & $\overline{D}$ & $\overline{l}$ & $D$ & $\overline{L}$ & $d$ \\
$l$ & $D$ & $L$ & $d$ & $L$ & $d$ & $l$ & $D$ \\
$l$ & $\overline{D}$ & $L$ & $\overline{d}$ & $\overline{L}$ & $d$ & $\overline{l}$ & $D$ \\ 
$l$ & $d$ & $l$ & $D$ & $L$ & $d$ & $L$ & $D$ \\
$l$ & $\overline{d}$ & $L$ & $\overline{D}$ & $\overline{l}$ & $d$ & $\overline{L}$ & $D$ \\
$\overline{L}$ & $D$ & $\overline{l}$ & $d$ & $L$ & $\overline{D}$ & $l$ & $\overline{d}$ \\
$\overline{L}$ & $\overline{D}$ & $\overline{l}$ & $\overline{d}$ & $\overline{L}$ & $\overline{D}$ & $\overline{l}$ & $\overline{d}$ \\ 
$\overline{L}$ & $d$ & $\overline{l}$ & $D$ & $l$ & $\overline{D}$ & $L$ & $\overline{d}$ \\
$\overline{L}$ & $\overline{d}$ & $\overline{l}$ & $\overline{D}$ & $\overline{l}$ & $\overline{D}$ & $\overline{L}$ & $\overline{d}$ \\
$\overline{l}$ & $D$ & $\overline{L}$ & $d$ & $L$ & $\overline{d}$ & $l$ & $\overline{D}$ \\
$\overline{l}$ & $\overline{D}$ & $\overline{L}$ & $\overline{d}$ & $\overline{L}$ & $\overline{d}$ & $\overline{l}$ & $\overline{D}$ \\ 
$\overline{l}$ & $d$ & $\overline{L}$ & $D$ & $l$ & $\overline{d}$ & $L$ & $\overline{D}$ \\
$\overline{l}$ & $\overline{d}$ & $\overline{L}$ & $\overline{D}$ & $\overline{l}$ & $\overline{d}$ & $\overline{L}$ & $\overline{D}$ \\
  \hline
\end{tabular}
}
\caption{}
\label{table3}
\end{table}

\item If $v_i$ is live with respect to $\gamma$ and there exists $j \in \mathcal{C}_{i+1}(\gamma)$ such that $j < i$, then Lemma \ref{lemma4} implies $w_k(\gamma) = w_k(\gamma_{v_i v_{i+1}}) = w_k(\gamma^{v_i v_{i+1}}) = w_k((\gamma^{v_i v_{i+1}})_{v_i v_{i+1}})$ for all $k \neq i,i+1$. Now according to Lemma \ref{lemma1}(3) and Lemma \ref{lemma2}(2)(3), we obtain that $v_i$ is live with respect to $(\gamma^{v_i v_{i+1}})_{v_i v_{i+1}}$ and dead with respect to both $\gamma_{v_i v_{i+1}}$ and $\gamma^{v_i v_{i+1}}$, while $v_{i+1}$ is dead with respect to $\gamma$,  $\gamma_{v_i v_{i+1}}$, $\gamma^{v_i v_{i+1}}$, and $(\gamma^{v_i v_{i+1}})_{v_i v_{i+1}}$.

In Table \ref{table4}, we state the $i$-th and $(i+1)$-th letters of the activity words of the four circuits in all possible cases  using Remark \ref{remark1} and Remark \ref{remark2}. In the eight cases marked by an "$\times$" in the table, we have 
\[
\begin{cases} 
    \mu_i(\gamma) \cdot \mu_{i+1}(\gamma) = -\mu_i(\gamma^{v_i v_{i+1}}) \cdot \mu_{i+1}(\gamma^{v_i v_{i+1}}) ,& \\
    \mu_i(\gamma_{v_i v_{i+1}}) \cdot \mu_{i+1}(\gamma_{v_i v_{i+1}}) = - \mu_i((\gamma^{v_i v_{i+1}})_{v_i v_{i+1}}) \cdot \mu_{i+1}((\gamma^{v_i v_{i+1}})_{v_i v_{i+1}}).
\end{cases}
\]

So in these eight cases, the Euler circuits $\gamma$ and $\gamma^{v_i v_{i+1}}$ do not contribute to $X_G(q)$, neither do $\gamma_{v_i v_{i+1}}$ and $(\gamma^{v_i v_{i+1}})_{v_i v_{i+1}}$ to $X_{G_{v_i v_{i+1}}}(q)$. 

In the remaining cases, we have 

\[
\begin{cases} 
    \mu_i(\gamma) \cdot \mu_{i+1}(\gamma) = \mu_i((\gamma^{v_i v_{i+1}})_{v_i v_{i+1}}) \cdot \mu_{i+1}((\gamma^{v_i v_{i+1}})_{v_i v_{i+1}}) ,& \\ 
   \mu_i(\gamma^{v_i v_{i+1}}) \cdot \mu_{i+1}(\gamma^{v_i v_{i+1}}) = \mu_i(\gamma_{v_i v_{i+1}}) \cdot \mu_{i+1}(\gamma_{v_i v_{i+1}}). 
\end{cases}
\]

Hence the result then follows.

\begin{table}[H]
  \centering
  \scriptsize{
 \begin{tabular}{|cc|cc|cc|cc|c}
\cline{1-8}
\multicolumn{2}{|c|}{$\gamma$} & \multicolumn{2}{c|}{$\gamma^{v_i v_{i+1}}$} & \multicolumn{2}{c|}{$\gamma_{v_i v_{i+1}}$} & \multicolumn{2}{c|}{$(\gamma^{v_i v_{i+1}})_{v_i v_{i+1}}$} \\ 
$w_i$ & $w_{i+1}$ & $w_i$ & $w_{i+1}$ & $w_i$ & $w_{i+1}$ & $w_i$ & $w_{i+1}$ \\ \cline{1-8}
$L$ & $D$ & $d$ & $d$ & $D$ & $D$ & $l$ & $d$ & $\times$ \\
$L$ & $\overline{D}$ & $d$ & $\overline{d}$ & $\overline{D}$ & $D$ & $\overline{l}$ & $d$ & \\ 
$L$ & $d$ & $d$ & $D$ & $d$ & $D$ & $L$ & $d$ & \\
$L$ & $\overline{d}$ & $d$ & $\overline{D}$ & $\overline{d}$ & $D$ & $\overline{L}$ & $d$ & $\times$ \\
$l$ & $D$ & $D$ & $d$ & $D$ & $d$ & $l$ & $D$ & \\
$l$ & $\overline{D}$ & $D$ & $\overline{d}$ & $\overline{D}$ & $d$ & $\overline{l}$ & $D$ & $\times$ \\ 
$l$ & $d$ & $D$ & $D$ & $D$ & $d$ & $L$ & $D$ & $\times$ \\
$l$ & $\overline{d}$ & $D$ & $\overline{D}$ & $\overline{d}$ & $d$ & $\overline{L}$ & $D$ & \\
$\overline{L}$ & $D$ & $\overline{d}$ & $d$ & $D$ & $\overline{D}$ & $l$ & $\overline{d}$ & \\
$\overline{L}$ & $\overline{D}$ & $\overline{d}$ & $\overline{d}$ & $\overline{D}$ & $\overline{D}$ & $\overline{l}$ & $\overline{d}$ & $\times$ \\ 
$\overline{L}$ & $d$ & $\overline{d}$ & $D$ & $d$ & $\overline{D}$ & $L$ & $\overline{d}$ & $\times$ \\
$\overline{L}$ & $\overline{d}$ & $\overline{d}$ & $\overline{D}$ & $\overline{d}$ & $\overline{D}$ & $\overline{L}$ & $\overline{d}$ & \\
$\overline{l}$ & $D$ & $\overline{D}$ & $d$ & $D$ & $\overline{d}$ & $l$ & $\overline{D}$ & $\times$ \\
$\overline{l}$ & $\overline{D}$ & $\overline{D}$ & $\overline{d}$ & $\overline{D}$ & $\overline{d}$ & $\overline{l}$ & $\overline{D}$ & \\ 
$\overline{l}$ & $d$ & $\overline{D}$ & $D$ & $d$ & $\overline{d}$ & $L$ & $\overline{D}$ & \\
$\overline{l}$ & $\overline{d}$ & $\overline{D}$ & $\overline{D}$ & $\overline{d}$ & $\overline{d}$ & $\overline{L}$ & $\overline{D}$ & $\times$ \\
\cline{1-8}
 \end{tabular}
 }
\caption{}
\label{table4}
\end{table}
\item If $v_i$ is dead with respect to $\gamma$ and $\mathcal{C}_{i+1}(\gamma) \setminus \left\lbrace i \right\rbrace \subset \left\lbrace i+2, \hdots , n \right\rbrace$, then  Lemma \ref{lemma1}(3) implies that $v_i$ is live with respect to $\gamma_{v_i v_{i+1}}$ and there exists $j \in \mathcal{C}_{i+1}(\gamma_{v_i v_{i+1}})$ such that $j < i$. The result follows from the previous case by interchanging the roles of $\gamma$ and $\gamma_{v_i v_{i+1}}$.

\item If $v_i$ is dead with respect to $\gamma$ and there exists $j \in \mathcal{C}_{i+1}(\gamma)$ such that $j < i$, then Lemma \ref{lemma1}(3) implies that $v_i$ is dead with respect to $\gamma_{v_i v_{i+1}}$, and  $w_i(\gamma)=w_{i+1}(\gamma_{v_i v_{i+1}}), w_{i+1}(\gamma)=w_i(\gamma_{v_i v_{i+1}})$. By Lemma \ref{lemma2}(1), we have that $\gamma$ and $\gamma_{v_i v_{i+1}}$ have the same weights.
\end{enumerate}
\end{enumerate}
\end{proof}

\begin{proof}[Proof of Theorem \ref{theorem2}]
We prove the first case and the same argument can be applied to obtain the second case. Now we label the vertices of $G_{v}$ as $v_1, \hdots , v_n$ such that $v=v_n$. We consider two subcases:
\begin{enumerate}
\item If $v$ is not a cut vertex that is a vertex whose removal does not increase the number of components, then it is not too hard to see that there is a natural one-to-one correspondence $\psi$ induced by the merging operations introduced in Section 3 between the set of Euler circuits of $G^v$ 
and the disjoint union of the sets of the Euler circuits of $G_0^v$ and $G_1^v$ 
In fact, if $\gamma$ is an Euler circuit of $G^v$ 
then $\psi(\gamma)$ is an Euler circuit of $G_0^v$ or $G_1^v$ by choosing the appropriate merging operation at the vertex $v$.  Also, if $\beta$ is an Euler circuit of $G_0^v$ or $G_1^v$, then $\psi^{-1}(\beta)$ is an Euler circuit of $G$ obtained by reversing the merging operation at the vertex $v$. Now the chord diagram $C(\psi(\gamma))$ of $\psi(\gamma)$ is obtained from the chord diagram $C(\gamma)$ simply by deleting the $n$-th chord. Thus, if $W(\gamma)=w_1(\gamma) \hdots w_n(\gamma)$ is the activity word of $\gamma$, then the activity word $W(\psi(\gamma))$ of $\psi(\gamma)$ is $w_1(\gamma) \hdots w_{n-1}(\gamma)$. Moreover,  it is easy to see that $w_n(\gamma) \in \left\lbrace D,d \right\rbrace$ since $v$ is equipped with a positive sign and $C_{n} \neq \emptyset$. According to the proof of Theorem \ref{main}, the result is independent of the choice of  coloring for $G^v$. Thus we can choose appropriate coloring of $G^v$ in a way that the $n$-th marker in the graph $G_0^v$ is an $A$-marker and the $n$-th marker in the graph $G_1^v$ is an $a$-marker. 
This implies that $w(\gamma) = qw(\psi(\gamma))$ or $w(\gamma) = q^{-1}w(\psi(\gamma))$, respectively. Thus the result follows directly from the summation over all Euler circuits of $G^v$.
\item If $v$ is a cut vertex, then one of the two graphs $G_0^v$ or $G_1^v$ consists of two components. Without loss of generality, we can assume that $G_0^v$ is connected and $G_1^v$ consists of two components denoted by $(G_1^v)'$ and $(G_1^{v})''$. Now it is not too hard to see that adding or removing  appropriately a new chord that corresponds to the vertex $v$ yields a one-to-one correspondence $\eta$ between the chord diagrams of $G^v$ and the chord diagrams of the graph $G_0^v$. Thus, if $W(\gamma)=w_1(\gamma) \hdots w_n(\gamma)$ is the activity word of $\gamma$, then the activity word $W(\eta(\gamma))$ of $\eta(\gamma)$ is $w_1(\gamma) \hdots w_{n-1}(\gamma)$. Moreover,  it is easy to see that $w_n(\gamma) \in \left\lbrace L,l \right\rbrace$ since $v$ is equipped with a positive sign and $C_{n}(\gamma) = \emptyset$. According to the proof of Theorem \ref{main}, the result is independent of the choice of coloring for $G^v$. Thus we can choose appropriate coloring of $G^v$ in a way that the $n$-th marker in the graph $G_0^v$ is an $A$-marker and the $n$-th marker in the graph $G_1^v$ is an $a$-marker. This implies that $X_{G^v}(q) = -q^{-3}X_{G_0^v}(q)$. Now we prove that the right-hand side has also the same value.
We can see that any chord diagram for the graph $G_0^v$ is a union of two chord diagrams in appropriate way of the graphs $(G_1^v)'$ and $(G_1^{v})''$ and vice-versa. This implies that  $X_{G_1^v}(q) = -(q^{2}+q^{-2}) X_{G_0^v}(q)$. Now we obtain $q  X_{G_0^v}(q) + q^{-1}  X_{G_1^v}(q) = qX_{G_0^v}(q) + q^{-1}(-q^{2}-q^{-2})X_{G_0^v}(q) = -q^{-3}X_{G_0^v}(q)$ as requested.
\end{enumerate}
\end{proof}

\end{document}